\documentclass[11pt, a4paper]{article}
\usepackage{geometry}
\usepackage[T1]{fontenc}
\usepackage{lmodern}
\usepackage{textcomp}
\usepackage[utf8]{inputenc}
\usepackage[english]{babel}
\usepackage{amsmath}
\usepackage{amssymb}
\usepackage{parskip}
\usepackage{amsthm}
\usepackage{booktabs}
\usepackage{graphicx}
\usepackage{sidecap}
\usepackage[bookmarksdepth=1]{hyperref}
\usepackage{xcolor}

\usepackage{pgfplots}
\pgfplotsset{compat=1.16}
\usetikzlibrary{math}

\newtheorem{theorem}{Theorem}
\newtheorem{proposition}{Proposition}
\newtheorem{lemma}[theorem]{Lemma}

\newtheorem{corollary}{Corollary}
\newtheorem{remark}{Remark}

\newcommand{\Sper}{S^{per}}                       

\newcommand{\dhx}{\,\mathrm d\widehat x}
\newcommand{\dx}{\, \mathrm dx}
\newcommand{\dy}{\,\mathrm dy}
\newcommand{\Bdy}{\mathcal{D}}

\graphicspath{{./figures/}}

\title{The Isogeometric Fast Fourier-based Diagonalization method}

\author{
Monica Montardini\footnote{University of Pavia, Via Ferrata 5, 27100, Pavia. \texttt{monica.montardini@unipv.it}},
Stefan Takacs\footnote{Institute of Numerical Mathematics, Johannes Kepler University Linz, Altenberger Str. 60, 4040 Linz, Austria. \texttt{stefan.takacs@numa.uni-linz.ac.at}},
Mattia Tani\footnote{University of Pavia, Via Ferrata 5, 27100, Pavia. \texttt{mattia.tani@unipv.it}}
}
\date{}

\begin{document}

\maketitle

\begin{abstract}
The construction of robust solvers for linear systems obtained from the discretization of partial differential equations using Isogeometric Analysis is challenging since the condition number of the system matrix not only grows with the reciprocal square of the grid size (for second order problems), but also exponentially with the spline degree.
The Fast Diagonalization method allows the construction of a preconditioner that is robust both in grid size and spline degree. Although this method is efficient in practice, its computational complexity is superlinear in the number of degrees of freedom. 
In this work, we construct a variant of the Fast Diagonalization method that can exploit the Fast Fourier Transformation. Note that, because of boundary effects, a Fourier Transformation cannot diagonalize the overall problem.
We circumvent this issue using a stable splitting of the spline space. The resulting preconditioner is still robust with respect to the grid size and spline degree and can be realized with a computational complexity that grows almost linearly with the number of degrees of freedom.
\end{abstract}

\section{Introduction}

It is known that when the Poisson problem is discretized on a tensor-product grid using piecewise linear finite elements (or equivalently centered finite differences), then the resulting linear system can be solved very efficiently using a so-called Fast Poisson solver. These methods, which were very popular when finite differences were the standard approach to tackle partial differential equations, rely on the Fast Fourier Transform or other ad-hoc techniques to compute the solution with $\mathcal{O} \left( N_{dof} \log N_{dof} \right) $ (or even less) floating point operations (flops), where $N_{dof}$ is the total number of degrees of freedom, see e.g.~\cite[Section 4.5]{VanLoan1992} and references therein. 

The purpose of this work is to extend Fast Poisson solvers to problems discretized with splines of arbitrary degree, in the context of isogeometric analysis (IGA)~\cite{CottrellHughesBazilevs2005}. Here, assuming that the physical domain $\Omega \subset \mathbb{R}^d$, $d \geq 2$, is given in parametrized form, the functions of the Galerkin space are taken as the pull-back of tensor product B-splines defined on the reference domain $[0,1]^d$.
In this context, iterative Krylov methods are often preferred over direct ones to address large linear systems, since the latter suffer from the lower sparsity associated with smooth basis functions, see~\cite{Collier2012}. However, it is particularly challenging to design a preconditioner that is robust with respect to the spline degree $p$. Indeed, multigrid methods with standard smoothers, that are suitable for low-order finite elements, are robust with respect to the mesh size $h$ but suffer from the exponential growth of the condition number with $p$, see~\cite{Gahalaut2013}. Carefully designed smoothers, like the one proposed in~\cite{Hofreither2017}, are required to obtain multigrid solvers that are robust with respect to both parameters.

An alternative approach is the Fast Diagonalization (FD) method~\cite{Lynch1964}, which has already been used in the context of IGA, see~\cite{Sangalli2016,Fuentes2023,Bahari2024}. This method, yielding a preconditioner that is robust with respect to both $h$ and $p$, is a direct solver for the problem on the reference domain. Here, the tensor structure of B-splines induces a tensor structure on the eigenvectors of the Galerkin matrix, which can be leveraged to solve the linear system by explicitly diagonalizing it. If no further structure is assumed, the complexity of one application of the preconditioner is independent of $p$ and proportional to $N_{dof}^{1 + 1/d}$. Even though the method is very fast in practice, the computational cost is superlinear and this could be considered an issue. 
In this paper, we propose to replace the exact eigendecomposition with an approximate one, allowing to exploit the Fast Fourier Transform (FFT) to apply the preconditioner.
Note that a Fourier transformation would be able to diagonalize the involved matrices if appropriate periodic boundary conditions were imposed. However, this is not the case if standard B-spline discretizations, as obtained with open knot vectors, are considered. 

Based on the results from~\cite{Hofreither2017}, for each parametric direction we decompose the spline space into a ``regular'' part and a small remainder, and we diagonalize the problem separately on each one. On the small subspace, the eigendecomposition can be computed numerically. As for the regular subspace, we show that the discrete eigenfunctions interpolate the eigenfunctions of the Laplacian associated with the proper boundary conditions, which to the best of our knowledge is a novel result for this specific setting and, we think, one interesting in itself. 
In particular, matrix-vector products involving the eigenvector matrix can be computed using a (fast) Fourier transformation, yielding an application cost for the preconditioner that is proportional to $N_{dof} \left( \log N_{dof} + p \right)$ flops. We refer to this approach as \textit{Isogeometric Fast Fourier-based Diagonalization} (IFFD) method.
Since the decomposition of the spline space is stable with respect to the relevant Sobolev norms, the IFFD preconditioner is spectrally equivalent to the FD preconditioner.
Furthermore, we show that for a particular choice of the basis for the regular subspace the expression of the discrete eigenfunctions simplifies, and the interpolation operator is not required. We remark that the latter result was already shown in the recent works~\cite{Deng2023,Lamsahel2025}. We also mention that in~\cite{Ekstroem2018} it was shown that, if $p=2$ and only Dirichlet boundary conditions are considered, then the system matrix can be diagonalized with the FFT. This is consistent with the results derived in this paper, since in that case the reminder subspace is null. Finally, we remark that the presented approach was already employed by the first and the last author in~\cite{Montardini2023,Montardini2025} as a preconditioner for low-rank Krylov methods. In these works, however, no theoretical justification was given.

This paper is organized as follows. In Section \ref{sec:model_problem}, we present the model problem and the isogeometric discretization. In Section \ref{sec:FD} we review the FD method and introduce its approximated version. In Section \ref{sec:splitting}, we present the splittings of the univariate spline spaces and show their stability in the $H^1$ norm, from which the robustness of the IFFD preconditioner is derived.
In Section \ref{sec:discrete_eigenfunctions}, we derive an explicit expression for the discrete eigenfunctions, while in Section \ref{sec:special_basis} we derive the simplified expression for the eigenvectors obtained by fixing a particular basis. Finally, in Section \ref{sec:num_res} we numerically assess the performance of our preconditioning strategy, and in Section \ref{sec:conclusions} we draw the conclusions, suggesting in particular how this strategy can be applied in a multi-patch setting, where the domain is given as the union of independently parametrized subdomains.

\section{Model problem and discretization}
\label{sec:model_problem}

In this paper, we discuss the main ideas for the following single-patch model problem. Let $\Omega \subset \mathbb R^d$, $d \geq 2$, be the computational domain and assume $\partial \Omega = \Gamma_\Bdy \cup \Gamma_{\mathcal N}$, where $\Gamma_\Bdy \cap \Gamma_{\mathcal N} = \emptyset$ and $\Gamma_\Bdy$ has positive measure.
Given $f : L^2(\Omega) \to \mathbb{R}$,
we consider the problem of finding  $u\in H^1_\Bdy(\Omega) := \left\{ v \in H^1(\Omega) \, : \,  v=0 \text{ on } \Gamma_\Bdy \right\}$ such that
\begin{equation}\label{eq:vf}
	\int_\Omega \nabla u \cdot \nabla v \dx
	=
	\int_\Omega fv \dx
	\qquad \mbox{for all} \qquad
	v\in H^1_D(\Omega).
\end{equation}
We assume that the computational domain is parameterized with a bijective geometry function $\mathbf G:\widehat\Omega:=[0,1]^d \to \Omega$, satisfying
\begin{equation}\label{eq:geo-equiv}
		\|J_{\mathbf G}\|_{L^\infty} \le C_1
		\qquad\mbox{and}\qquad
		\|J_{\mathbf G}^{-1}\|_{L^\infty} \le C_2
\end{equation}
with suitable constants, where $J_{\mathbf G}$ denotes the Jacobian. Using the \emph{pull-back principle}, the problem can be entirely represented using integrals on the parameter domain $\widehat\Omega$. Let
$\widehat u := u \circ \mathbf G$ and $\widehat f := f \circ \mathbf G$ be the representations of the solution and the right-hand side on the parameter domain, respectively.
Analogously, let $\widehat \Gamma_\Bdy = \mathbf G^{-1}(\Gamma_\Bdy)$ be the pull-back of the Dirichlet boundary. We assume that $\widehat \Gamma_\Bdy$ is the union of one or more faces (edges for $d=2$ or faces for $d=3$) of $\widehat \Omega$, i.e.,
there are sets $\Bdy_1,\dots,\Bdy_d \subseteq \left\{ 0,1 \right\} $ such that
\begin{equation}\label{eq:gramma tens}
    \widehat \Gamma_\Bdy
    =
    \bigcup_{k=1}^d
    \{ (x_1,\dots,x_d) \in \widehat\Omega
        : x_k \in \Bdy_k 
    \}.
\end{equation}
Using standard chain and substitution rules, the problem~\eqref{eq:vf} can be equivalently rewritten as follows. Find $\widehat u\in H^1_\Bdy(\widehat \Omega):=\{ v \in H^1(\widehat \Omega) \, : \,  v=0 \text{ on } \widehat \Gamma_D \}$ such that
\[
	\underbrace{\int_{\widehat\Omega} (J_{\mathbf G}^{-\top} \nabla \widehat u) \cdot  (J_{\mathbf G}^{-\top}\nabla \widehat v) \, |\det  J_{\mathbf G}| \dhx}_{\displaystyle a(\widehat u, \widehat v) := }
	=
	\underbrace{\int_{\widehat\Omega} \widehat f \widehat v \, |\det  J_{\mathbf G}| \dhx}_{\displaystyle \langle \widehat f, \widehat v\rangle := }
	\qquad \mbox{for all} \qquad
	\widehat v\in H^1_\Bdy(\widehat \Omega).
\]
This problem is discretized using tensor-product splines. For simplicity, we restrict ourselves to splines of maximum smoothness, although the proposed approach can be generalized to less regular splines, see Remark \ref{rmk:reduced_regularity}. 
The space of such splines can be described by the breakpoints $Z=(z_0,z_1,\ldots z_n)$ with $0=z_0 < z_1 <\cdots <z_{n-1}<z_n=1$ via the formula
\[
		S_{p,h}
		:=
		\left\{
				v_h \in C^{p-1}[0,1]
				\, : \,
				v_h|_{[z_{k-1},z_k]} \in \pi_p \;\; \mbox{for} \;\; k \in \{1,\ldots,n\}
		\right\},
\]
where $\pi_p$ denotes the space of polynomial functions of degree (at most) $p$. We assume that the grid of breakpoints is uniform, i.e. $z_k = kh$ for $k=0,\ldots,n$, with $h = 1/n$.

The standard basis for this space is defined by the Cox-de Boor formula, see e.g. \cite{Boor2001}. That basis is advantageous since it forms a partition of unity and each basis function is non-negative and has a bounded support. For using the Cox-de Boor formula, we represent the space using a $p$-open knot vector $\Xi = (\xi_{0},\xi_1,\ldots \xi_{n+2p})$, which satisfies $0 = z_0 = \xi_{0}=\cdots= \xi_p < \xi_{p+1}=z_1 < \cdots < \xi_{n+p-1} = z_{n-1} < \xi_{n+p} = \cdots  = \xi_{n+2p} = z_n = 1$ for the case of maximum smoothness. The basis functions $\mathcal B_{p,h,0},\ldots \mathcal B_{p,h,n+p-1}$ are recursively given by
\begin{align*}
\mathcal B_{0,h,i}(x) & :=
\begin{cases}
1 & \mbox{if }  x \in [ \xi_i, \xi_{i+1} ) \\
0 & \mbox{otherwise}
\end{cases}
&&  \hspace{-3em}\mbox{for } i=0,\ldots, n+2p-1, \\
\mathcal B_{q,h,i}(x) 
& := \frac{ x - \xi_i }{\xi_{i+q} - \xi_i} \mathcal B_{q-1,h,i}(x) +\frac{\xi_{i+q+1}-x }{\xi_{i+q+1} - \xi_{i+1}} \mathcal B_{q-1,h,i+1}(x)
\\&&&\hspace{-3em} \mbox{for } i=0,\ldots, n+2p-q-1.
\end{align*}
In any case where the denominator vanishes in the recursion formula, also $\mathcal B_{p-1,h,i}$ or $\mathcal B_{p-1,h,i+1}$ vanishes, respectively, and the corresponding summand is dropped.  The results obtained for $q=p$, give the B-spline basis
\begin{align}\label{eq:bspline basis}
    \{\mathcal  B_{p,h,i}(x) : i=0,\dots,n+p-1 \}    .
\end{align}

For ease of notation, we restrict ourselves to the two-dimensional case, and consider the same mesh size for both directions.
As function space over the parameter domain $\widehat \Omega$, we consider
\[
	\widehat V_h := \{ \widehat v_h \in S_{p,h} \times  S_{p,h} : \widehat v_h |_{\mathbf G^{-1} \left( \Gamma_\Bdy \right) } = 0 \}.
\]
Since we have assumed in~\eqref{eq:gramma tens} that $\widehat \Gamma_\Bdy$ is the union of one or more edges of the unit square, $\widehat V_h$ is a tensor product space
\[
	\widehat V_h = S_{p,h,\Bdy_1} \times  S_{p,h,\Bdy_2},
\]
where, recalling that $\Bdy_k \subseteq \left\{ 0,1\right\}$ we have defined
\[
	S_{p,h,\Bdy_k} := \{ \widehat v_h \in S_{p,h}  : \widehat v_h(x) = 0 \;\; \mbox{for} \;\; x \in \Bdy_k\}.
\]

For $k=1,2$, we fix a basis for $S_{p,h,\Bdy_k}$ by removing from the standard basis~\eqref{eq:bspline basis} the functions that do not vanish on $\Bdy_k$. Since B-splines interpolate at the boundary, this procedure removes $|\Bdy_k|=\dim S_{p,h} - \dim S_{p,h,\Bdy_k}$ basis functions. We refer to the chosen basis functions as $B_{1,k}, \ldots,B_{m_k,k}$, where $m_k$ denotes the dimension of $S_{p,h,\Bdy_k}$. Then we introduce the  lexicographically ordered basis for $\widehat V_h $ as
\begin{equation}\label{eq:tp-B-basis}
		{\mathbf B}_{i+m_1(j-1)}( x_1, x_2) := B_{i,1}(x_1) B_{j,2}(x_2),
\end{equation}
for $i=1,\ldots,m_1$ and $j=1,\ldots,m_2$. Let $N_{dof} = m_1 m_2$ denote the dimension of $\widehat V_h$.

The discretization of the model problem using the Galerkin principle reads as follows.
Find $\widehat u_h \in \widehat V_h$ such that
\[
	 a(\widehat u_h, \widehat v_h)
	=\langle \widehat f, \widehat v_h\rangle
	\quad\mbox{for all}\quad
	\widehat v_h \in \widehat V_h.
\]
The solution $\widehat u_h$ can be expressed by a coordinate vector $\mathbf u$, representing the solution with respect to the chosen tensor-product basis~\eqref{eq:tp-B-basis}. This gives rise to the matrix-vector formulation of the Galerkin discretization
\begin{equation}\label{eq:linsys}
	A\, \mathbf u = \mathbf f,
\end{equation}
where the $A := [ a(\mathbf B_j, \mathbf B_i) ]_{i,j=1}^{N_{dof}}$ is the stiffness matrix and
$\mathbf f := [ \langle \widehat f,  \mathbf B_i \rangle ]_{i=1}^{N_{dof}}$ is the load vector.

\section{Fast Diagonalization method revisited}
\label{sec:FD}

We are interested in solving~\eqref{eq:linsys} using an iterative solver. Since the system matrix is symmetric positive definite, we rely on the Conjugate Gradient method with an appropriate preconditioner. Thus, we consider the modified problem
\[
			\mathcal P^{-1} A \, \mathbf u = \mathcal P^{-1} \mathbf f
\]
for some appropriately chosen matrix $\mathcal P$. The convergence rate of the Preconditioned Conjugate Gradient (PCG) method is bounded by
\[
		\frac{\sqrt{\kappa(\mathcal P^{-1} A)}-1}{\sqrt{\kappa(\mathcal P^{-1} A)}+1},
\]
where $\kappa$ denotes the condition number of the corresponding matrix, intended as the ratio between its maximum and minimum eigenvalue. In order to obtain a \emph{robust} convergence, we need a uniform bound on $\kappa \left( \mathcal P^{-1} A \right)$. 
To realize the solver efficiently, we must be able to efficiently compute the application of the preconditioner to a vector, i.e., the results of $\mathcal P^{-1}\mathbf r$ for any given vector $\mathbf r$. As already announced in the introduction, we discuss the Fast Diagonalization (FD) method first.

To set up the FD method, we first approximate the stiffness matrix $A$ by a simpler matrix $\widehat A$ that is the sum of Kronecker product matrices. That matrix is obtained by discretizing a Poisson problem where the geometry map $\mathbf{G}$ is the identity: we define
$\widehat A := [ \widehat a(\mathbf B_j, \mathbf B_i) ]_{i,j=1}^{N_{dof}}$, where
\[
			\widehat a(\widehat u, \widehat v) :=
			\int_{\widehat \Omega} \nabla \widehat u(\widehat x) \cdot \nabla \widehat v(\widehat x) \dhx.
\]
The condition number of $\widehat A^{-1} A$ is bounded independent of the grid size $h$ and the spline degree $p$, specifically
$$\kappa( \widehat{A}^{-1} A) \leq \frac{ \sup |\det  J_{\mathbf G}| \sigma_{\max}^2 \left( J_{\mathbf G}^{-1} \right) }{\inf |\det  J_{\mathbf G}| \sigma_{\min}^2 \left( J_{\mathbf G}^{-1} \right) }, $$
where $\sigma_{\max}\left(\cdot \right)$ and $\sigma_{\min}\left(\cdot \right)$ denote the maximum and minimum singular values of a matrix, see, e.g. \cite{Sangalli2016}. The above inequality means that $\widehat A$ is a robust preconditioner for the problem~\eqref{eq:linsys} with respect to the spline degree $p$ and the grid size $h$. In the remainder of this section, we show how the application of $\widehat A^{-1}$ to a vector can be realized efficiently.

First, we observe that the tensor-product structure yields
\begin{align*}
	 & \widehat A_{i+m_1(s-1),j+m_1(t-1)}
			= \int_{\widehat \Omega} \nabla \mathbf B_{j+m_1(t-1)}(\widehat x) \cdot \nabla \mathbf B_{i+m_1(s-1)}(\widehat x) \dhx \\
			& =
			\int_0^1\int_0^1 \big( \tfrac{\partial}{\partial x} B_{j,1}(x) B_{t,2}(y) \tfrac{\partial}{\partial x} B_{i,1}(x) B_{s,2}(y)
			+
			\tfrac{\partial}{\partial y} B_{j,1}(x) B_{t,2}(y) \tfrac{\partial}{\partial y}  B_{i,1}(x) B_{s,2}(y) \big) \dx \dy \\
			&  =
			\underbrace{\int_0^1 B_{j,1}'(x)  B_{i,1}'(x) \dx}_{\displaystyle = [K_1]_{i,j}}
			\underbrace{\int_0^1 B_{t,2}(y)   B_{s,2}(y) \dy}_{\displaystyle = [M_2]_{s,t}}
			+
			\underbrace{\int_0^1 B_{j,1}(x)  B_{i,1}(x) \dx}_{\displaystyle = [M_1]_{i,j}}
			\underbrace{	\int_0^1 B_{t,2}'(y)   B_{s,2}'(y) \dy}_{\displaystyle = [K_2]_{s,t}},
\end{align*}
where, for $k=1,2,$ $M_k$ and $K_k$ represent the mass and stiffness matrices associated with the univariate spline space $S_{p,h,\Bdy_k}$. These matrices are symmetric. In addition to that, $M_k$ is always positive definite, while $K_k$ is positive definite unless $\Bdy_k = \emptyset$, in which case it is positive semidefinite and singular.
The above relation can be written using the Kronecker product as follows:
\begin{equation} \label{eq:system_matrix}
			\widehat A = K_1 \otimes M_2 + M_1 \otimes K_2.
\end{equation}
This matrix is diagonalized as follows. For $k=1,2$, one computes a generalized eigenvalue decomposition, i.e., one finds an $M_k-$orthonormal basis of eigenvectors $\mathbf q_1^{(k)},\ldots,\mathbf q_{m_k}^{(k)}$ and the corresponding eigenvalues $\lambda_1^{(k)},\ldots,\lambda_{m_k}^{(k)}$, i.e., such that
\begin{equation}\label{eq:eigendecomposition}	K_k \mathbf q_j^{(k)} = \lambda_j^{(k)} M_k \mathbf q_j^{(k)} \qquad \mbox{and}\qquad \mathbf q_j^{(k)} \ne 0. \end{equation}
and
\[ \left( \mathbf q_j^{(k)} \right)^\top M_k \mathbf q_i^{(k)} = \delta_{i,j}, \qquad \mbox{and}\qquad \left(\mathbf q_j^{(k)}\right)^\top K_k \mathbf q_i^{(k)} = \lambda_i^{(k)} \delta_{i,j}, \]
where $\delta_{i,j}$ is the Kronecker delta ($\delta_{i,i}=1$ and $\delta_{i,j}=0$ for $i\ne j$). 
By collecting the eigenvectors and the eigenvalues into matrices
\[ Q_k := ( \mathbf q_1^{(k)}\; \mathbf q_2^{(k)}\; \cdots \mathbf q_{m_k}^{(k)} )\in \mathbb R^{m_k\times m_k}, \qquad \Lambda_k := \begin{pmatrix} \lambda_1^{(k)} \\&\ddots \\&&\lambda_{m_k}^{(k)}\end{pmatrix},\]
we can diagonalize the mass and stiffness matrices:
\begin{equation} \label{eq:eigedecomposition}
		Q^\top_k M_k Q_k = I,
		\qquad
		Q_k^\top K_k Q_k = \Lambda_k.
\end{equation}
Consequently, we have
\[ 
		\widehat A
		=
		(Q_2^{-\top} \otimes Q_1^{-\top} )
		(
			\Lambda_2 \otimes I + I \otimes \Lambda_{1}
		)
		(Q_2^{-1} \otimes Q_1^{-1})
\]
and therefore
\begin{equation} \begin{aligned} \label{eq:exact_diagonalization}
		\widehat A^{-1}
		& =
		(Q_2 \otimes Q_1)
		(
			\Lambda_2 \otimes I + I \otimes \Lambda_1
		)^{-1}
		(Q_2^\top \otimes Q_1^\top) \\
		& =
		(Q_2^\top \otimes I)(I \otimes Q_1^\top)
		(
			\Lambda_2 \otimes I + I \otimes \Lambda_1
		)^{-1}
		(I \otimes Q_1)(Q_2 \otimes I).
\end{aligned} \end{equation}
For the application of $\widehat A^{-1}$ to some vector $\mathbf r$, one has first to determine
$(Q_2 \otimes I)\mathbf r$, which is nothing else than multiplying $m_1$ chunks of $\mathbf r$ with the matrix $Q_2$. Since $Q_2$ is a dense $m_2\times m_2$ matrix, these $m_2$ matrix-vector multiplications can be realized with a computational complexity of $\mathcal O(m_2^2 m_1)$. Similarly, the complexity of multiplication with $(I \otimes Q_1)$ is $\mathcal O(m_2 m_1^{2})$, and the multiplications with their transposes are completely analogous. Note also that the multiplication with $(\Lambda_2 \otimes I + I \otimes \Lambda_1 )^{-1}$ is just a diagonal scaling.

In conclusion, since $m_1 \approx m_2$ the application of the FD preconditioner can be realized with a computational complexity of $\mathcal O\left(N_{dof}^{3/2}\right)$ in the two-dimensional case, or $\mathcal O\left( N_{dof}^{1+1/d} \right)$ for the general $d$-dimensional case. While this approach is fast in practice, it is theoretically suboptimal since its complexity does not scale linearly with respect to the number of unknowns $N_{dof}$.

In light of this, we aim at replacing the exact diagonalization by an approximate one, whose exact choice will be discussed in the next section. Here, for $k=1,2$, we generically consider a matrix $\widetilde{Q}_k \approx Q_k$ that we assume to be $M_k-$orthogonal, i.e.
\begin{equation} \label{eq:orthogonality} \widetilde{Q}_k^\top M_k \widetilde{Q}_k = I \end{equation}
and a diagonal matrix $\widetilde{\Lambda}_k \approx \Lambda_k$. The resulting preconditioner is then
\begin{equation} \label{eq:preconditioner} \mathcal P^{-1} = \left( \widetilde{Q}_2 \otimes \widetilde{Q}_1 \right) \left( \widetilde{\Lambda}_2 \otimes I + I \otimes \widetilde{\Lambda}_1 \right)^{-1} \left( \widetilde{Q}_2^{\top} \otimes \widetilde{Q}_1^{\top} \right) 
\end{equation}
The quality of such preconditioner is characterized in the next result.
\begin{proposition} \label{prop:condition_bound}
For $k=1,2,$ let $\widetilde{K}_k= \widetilde{Q}_k^{-\top} \widetilde{\Lambda}_k \widetilde{Q}_k^{-1}$ and assume there are constants $c_k,C_k > 0$ such that
$$ c_k \, \mathbf v^\top \widetilde{K}_k \mathbf v \leq \mathbf v^\top K_k \mathbf v \leq C_k \,  \mathbf v^\top \widetilde{K}_k \mathbf v \quad\mbox{for all}\quad \mathbf v \in \mathbb{R}^{m_k}. $$
Then it holds
$$ \kappa\left(\mathcal P^{-1} \widehat{A}_h \right) \leq \max_{k=1,2}  \frac{C_k}{c_k} $$
\end{proposition}
\begin{proof}
From \eqref{eq:orthogonality} we infer that $M_k = \widetilde{Q}^{-\top}_k \widetilde{Q}_k^{-1}$ for $k=1,2$.
Then inverting \eqref{eq:preconditioner} we obtain
\begin{equation} \label{eq:preconditioner2} \mathcal P = \widetilde{K}_2 \otimes M_1 + M_2 \otimes \widetilde{K}_1 \end{equation}
Since all involved matrices are symmetric, thanks to the properties of the Kronecker product it holds
$$ c_1 \, \mathbf v^\top  \left( M_2 \otimes \widetilde{K}_1 \right) \mathbf v \leq \mathbf v^\top \left( M_2 \otimes K_1 \right) \mathbf v \leq C_1 \, \mathbf v^\top \left( M_2 \otimes \widetilde{K}_1 \right) \mathbf v \quad\mbox{for all}\quad \mathbf v \in \mathbb{R}^{N_{dof}},  $$
and similarly
$$ c_2 \, \mathbf v^\top \left( \widetilde{K}_2 \otimes M_1 \right) \mathbf v \leq \mathbf v^\top \left( K_2 \otimes M_1 \right) \mathbf v \leq C_2 \,\mathbf v^\top \left( \widetilde{K}_2 \otimes M_1 \right)\mathbf v \quad\mbox{for all}\quad\mathbf v \in \mathbb{R}^{N_{dof}}. $$
These inequalities can be used to bound from above and below the Rayleigh quotient $ \left(\mathbf v^\top \widehat{A}_h \mathbf v \right) \slash \left( \mathbf v^\top \mathcal P \mathbf v \right)$, for $\mathbf v \in \mathbb{R}^{N_{dof}}$, from which we deduce the desired result.
\end{proof}

\section{Approximate eigendecomposition based on space splitting}
\label{sec:splitting}

In this section, we discuss how to construct an approximate diagonalization for the mass and stiffness matrices on the univariate spline space $S_{p,h,\Bdy_k}$, for $k=1,2$. For ease of notation, in the following we drop the subscript $k$ that identifies the parametric direction.

We consider the following splitting of the spline space \cite{Takacs2016,Hofreither2017}
\begin{equation}\label{eq:splitting} 
    S_{p,h,\Bdy} = S_{p,h,\Bdy}^{reg} \oplus S_{p,h,\Bdy}^{out},
\end{equation}
where
\begin{equation}  \label{eq:reg_space} 
    S_{p,h,\Bdy}^{reg} =  \left\{ v_h \in S_{p,h,\Bdy} :
    \begin{array}{l}\partial^{2k} v_h(x) = 0 \mbox{ for } x \in \Bdy,  \; k = 1,\ldots, \lfloor \tfrac{p-1}{2} \rfloor,  \\ 
    \partial^{2k+1} v_h(x) = 0 \mbox{ for } x \in \{0,1 \} \setminus \Bdy, \; k = 0,\ldots, \lfloor \tfrac{p}{2} \rfloor.
    \end{array}\right\},
\end{equation}
$\lfloor \cdot \rfloor$  denotes the integer part of a number, and $ S_{p,h,\Bdy}^{out} = \left( S_{p,h,\Bdy}^{reg} \right)^{\perp L^2} $. In other words, $S_{p,h,\Bdy}^{reg}$ is the subspace of $S_{p,h,\Bdy}$ with null even derivatives (up to order $p-1$) at the extremes associated with Dirichlet boundary conditions, and null odd derivatives (again, up to order $p-1$) at the other extremes.
Defining $n_{reg} := \dim \left( S_{p,h}^{reg} \right)$
and $n_{out} := \dim \left( S_{p,h}^{out} \right)$, we observe that
\begin{equation} \label{eq:reg_space_dim} n_{reg} = 		\begin{cases}
		n-1 & \mbox{if } p \mbox{ is odd, } \Bdy = \{ 0,1\}, \\
		n+1 & \mbox{if } p \mbox{ is odd, } \Bdy = \emptyset, \\
        n & \mbox{otherwise.} \\
		\end{cases},  \quad n_{out} = \begin{cases}
		p-2 & \mbox{if } p \mbox{ is even, } \Bdy = \{ 0,1\}, \\
		p & \mbox{if } p \mbox{ is even, } \Bdy = \emptyset, \\
        p-1 & \mbox{otherwise.} \\
		\end{cases}
\end{equation}
The space $S_{p,h,\Bdy}^{reg}$ represents the large, regular part of $S_{p,h,\Bdy}$. It has been observed (see \cite{Hiemstra2021,Deng2023,Lamsahel2025}) that $S_{p,h,\Bdy}^{reg}$ is not affected by the presence of spurious eigenpairs that are common to IGA  discretizations \cite{Cottrell2006}. These outliers are therefore contained in its orthogonal complement $S_{p,h,\Bdy}^{out}$.
Note that if $p=1$, or if $p=2$ and $\Bdy = \{ 0,1\}$, then 
$S_{p,h,\Bdy} = S_{p,h,\Bdy}^{reg}$.
Since the splitting is $L^2$-orthogonal, we have
\[
    \| v_h \|_{L^2(0,1)}^2 = \| v_{h}^{reg} \|_{L^2(0,1)}^2 + \| v_{h}^{out} \|_{L^2(0,1)}^2.
\]
The next theorem guarantees that the splitting is also stable in the $H^1$-seminorm.

\begin{theorem} \label{thm:stable_splitting}
There exist constants $c, C > 0 $ independent of $h$, $p$, and $\Bdy$ such that for every $v_h \in S_{p,h,\Bdy}$ 
$$ c \left( \vert v_{h}^{reg} \vert_{H^1(0,1)}^2 + \vert v_{h}^{out} \vert_{H^1(0,1)}^2 \right) \leq \vert v_h \vert_{H^1(0,1)}^2 \leq C \left( \vert v_{h}^{reg} \vert_{H^1(0,1)}^2 + \vert v_{h}^{out} \vert_{H^1(0,1)}^2 \right) $$
holds
with $v_{h}^{reg} \in S_{p,h,\Bdy}^{reg} $ and  $v_{h}^{out} \in S_{p,h,\Bdy}^{out} $ such that $v_h = v_{h}^{reg} + v_{h}^{out}$.
\end{theorem}
A proof for the case $\Bdy=\emptyset$ has already been given in \cite[Theorem 4]{Hofreither2017}. For completeness, we give a proof for general $\Bdy$ in the Appendix.

Let $V_{reg} \in \mathbb{R}^{m \times n_{reg}}$ and $V_{out} \in \mathbb{R}^{m \times n_{out}}$ be matrices whose columns form a coordinate basis for $S_{p,h,\Bdy}^{reg}$ and $S_{p,h,\Bdy}^{out}$, respectively, with respect to the standard basis of $S_{p,h,\Bdy}$. We restrict the mass and the stiffness matrices on $S_{p,h,\Bdy}^{reg}$ and $S_{p,h,\Bdy}^{out}$ and consider their generalized eigendecompositions
\begin{equation} \label{eq:eigendecomposition_reg} Q_{reg}^{\top} \left(V_{reg}^{\top} K V_{reg}\right) Q_{reg} = \Lambda_{reg}, \qquad Q_{reg}^{\top} \left(V_{reg}^{\top} M V_{reg}\right) Q_{reg} = I \end{equation}
\begin{equation} \label{eq:eigendecomposition_out} Q_{out}^{\top} \left(V_{out}^{\top} K V_{out}\right) Q_{out}  = \Lambda_{out}, \qquad Q_{out}^{\top} \left(V_{out}^{\top} M V_{out}\right) Q_{out} = I \end{equation}
where $ Q_{reg} \in \mathbb{R}^{n_{reg} \times n_{reg}}$ and $ Q_{out} \in \mathbb{R}^{n_{out} \times n_{out}}$ are the (normalized) eigenvector matrices, while $ \Lambda_{reg} \in \mathbb{R}^{n_{reg} \times n_{reg}}$ and $ \Lambda_{out} \in \mathbb{R}^{n_{out} \times n_{out}}$ are the diagonal matrices of eigenvalues.

Then the eigendecomposition \eqref{eq:eigendecomposition} is approximated as follows:
$$ Q \approx \widetilde{Q} := \begin{bmatrix} V_{reg} & V_{out} \end{bmatrix} \begin{bmatrix} Q_{reg} & 0 \\ 0 & Q_{out} \end{bmatrix} 
\quad\text{ and }\quad \Lambda \approx \widetilde{\Lambda} = \begin{bmatrix} \Lambda_{reg} & 0 \\ 0 & \Lambda_{out} \end{bmatrix}.  $$

If this approximate eigendecomposition is considered for the preconditioner \eqref{eq:preconditioner}, then the bound on the condition number given by Proposition \ref{prop:condition_bound} is robust with respect to $h$ and $p$. This is a direct consequence of Theorem \ref{thm:stable_splitting}, as formalized in the next result.

\begin{corollary}
Let $\widetilde{K} = \widetilde{Q}^{-\top} \widetilde{\Lambda} \widetilde{Q}^{-1}$. Then it holds
$$ c \, \mathbf v^{\top} \widetilde{K} \mathbf v \leq \mathbf v^{\top} K \mathbf v \leq C \, v^{\top} \widetilde{K} \mathbf v \quad \mbox{for all} \quad \mathbf  v \in \mathbb{R}^m ,$$
where $c$ and $C$ are the constants appearing in Theorem \ref{thm:stable_splitting}.
\end{corollary}
\begin{proof}
Given $\mathbf  v \in \mathbb{R}^{m},$ we write $ \begin{bmatrix} V_{reg} & V_{out} \end{bmatrix}^{-1}\mathbf v = \begin{bmatrix} \mathbf v_{reg} \\ \mathbf v_{out} \end{bmatrix},$ where $\mathbf v_{reg} \in \mathbb{R}^{n_{reg}}$ and $\mathbf u_{out} \in \mathbb{R}^{n_{out}}$ satisfy $\mathbf v = V_{reg} \mathbf v_{reg} + V_{out} \mathbf v_{out} $. Then it holds
\begin{align*} \mathbf v^{\top} \widetilde{K} \mathbf v & = \mathbf v_{reg}^{\top} Q_{reg}^{-\top} \Lambda_{reg} Q_{reg}^{-1} \mathbf v_{reg} + \mathbf v_{out}^{\top} Q_{out}^{-\top} \Lambda_{out} Q_{out}^{-1} \mathbf v_{out} \\ & = \mathbf v_{reg}^{\top} \left( V_{reg}^{\top} K  V_{reg} \right) \mathbf v_{reg} + \mathbf v_{out}^{\top} \left( V_{out}^{\top} K V_{out} \right) \mathbf v_{out}, \end{align*}
where the last equality derives from \eqref{eq:eigendecomposition_reg} and \eqref{eq:eigendecomposition_out}. Since $V_{reg} \mathbf v_{reg}$ and $V_{out} \mathbf v_{out}$ represent functions belonging to $S_{p,h,\Bdy}^{reg}$ and $S_{p,h,\Bdy}^{out}$ respectively, the desired result follows from Theorem \ref{thm:stable_splitting}.
\end{proof}

We conclude this section with some details on the numerical realization of the preconditioner.
We first observe that B-splines that are supported on $p+1$ elements have $p-1$ vanishing derivatives at the extremes of their support (since they are $C^{p-1}$ at those points) and thus belong to $S_{p,h,\Bdy}^{reg}$. Therefore, if $n > p$ we can assume that the matrix $V_{reg}$ has the block structure 
$$ V_{reg} = \begin{bmatrix} V_{reg,0} & 0 & 0 \\ 0 & I & 0 \\ 0 & 0 &  V_{reg,1} \end{bmatrix}, $$
with $V_{reg,0}, V_{reg,1} \in \mathbb{R}^{\lfloor (p-1)/2 \rfloor \times p}$ (we will discuss a specific choice for the basis of $S_{p,h,\Bdy}^{reg}$ in Section \ref{sec:special_basis}).
Then the $L^2$ orthogonality condition $V_{reg}^{\top} M V_{out} = 0$ implies that we can choose
$$ V_{out} = M^{-1} \begin{bmatrix} V_{reg,0}(V_{reg,0}^{\top} V_{reg,0})^{-1} & 0 \\ 0 & 0 \\ 0 &  V_{reg,1}(V_{reg,1}^{\top} V_{reg,1})^{-1} \end{bmatrix}. $$
Note that the above formula can be used to compute $V_{out}$. Then, one can compute $V_{out}^{\top} K V_{out}$ and its eigendecomposition \eqref{eq:eigendecomposition_out}. The total cost of these operations is $\mathcal{O}(n p^2) = \mathcal O(m p^2)$ flops.

In the next sections, we will discuss an explicit expression for the matrix $Q_{reg}$ that allows to exploit the Fast Fourier Transform for computing matrix-vector products, requiring only $ \mathcal{O} \left( m \left( \log m + p \right) \right) $ flops (which can be reduced to $ \mathcal O \left( m \log m \right)$ for a particular choice of the basis, see Section \ref{sec:special_basis}). Moreover, matrix-vector products involving the dense matrix $V_{out}$ require $\mathcal O(mp)$ flops, while matrix-vector products with $V_{reg}$ and $Q_{out}$ have negligible costs.
In conclusion, a single matrix-vector product involving $\widetilde{Q}$ requires $\mathcal O \left( m \left( \log m + p\right)\right)$ flops. As a consequence, by exploiting the properties of Kronecker products, a single application of the preconditioner \eqref{eq:preconditioner} can be computed at the (almost) optimal cost $\mathcal O \left( N_{dof} \left( \log N_{dof} + p\right)\right)$ flops. We refer to this approach as Isogeometric Fast Fourier-based Diagonalization (IFFD) method.

\begin{remark} \label{rmk:reduced_regularity}
Up to this point, we have only considered splines that are $C^{p-1}$ everywhere in the parameter domain, i.e. there are no repeated internal knots.
However, the geometry parametrization is often less regular (e.g. $C^0$) on specific manifolds, making it natural to reduce the regularity of the univariate spline spaces at some internal knots.
In this case, we can still consider the space splitting \eqref{eq:splitting}, where, however
\begin{equation}\label{eq:reg_space2} 
S_{p,h,\Bdy}^{reg} =  \left\{
v_h \in S_{p,h,\Bdy} \cap C^{p-1}
:
\begin{array}{l}
\partial^{2k} v_h(x) = 0 \mbox{ for } x \in \Bdy,  \; k = 0,\ldots, \lfloor \tfrac{p-1}{2} \rfloor, \\ 
\partial^{2k+1} v_h(x) = 0 \mbox{ for } x \in \{0,1 \} \setminus \Bdy, \; k = 0,\ldots, \lfloor \tfrac{p}{2} \rfloor.
\end{array}
\right\},
\end{equation}
while as before $ S_{p,h,\Bdy}^{out} = \left( S_{p,h,\Bdy}^{reg} \right)^{\perp L^2} $.
 We emphasize that the difference between definitions \eqref{eq:reg_space} and \eqref{eq:reg_space2} is that in the latter the functions are restricted to be globally $C^{p-1}$. 
This means in particular that $S_{p,h,\Bdy}^{reg}$ does not change if repeated knots are added to $S_{p,h,\Bdy}$.
On the other hand, for each repeated internal knot that is added to $S_{p,h,\Bdy}$, the dimension of $S_{p,h,\Bdy}^{out}$ increases by one. It follows that the numerical diagonalization \eqref{eq:eigendecomposition_out}, whose cost is cubic in the dimension of $S_{p,h,\Bdy}^{out}$, is feasible only if the number of repeated knots is small or moderate. However, this is clearly not an issue if the spline space has maximum regularity outside of a few points. The effectiveness of this splitting approach is validated numerically in Section \ref{sec:num_res}. 
\end{remark}

\section{Discrete eigenfunctions on the regular space}
\label{sec:discrete_eigenfunctions}

We now derive an expression for the discrete eigenfunctions of the Laplace operator on the regular space $S_{p,h,\Bdy}^{reg}$. As a preliminary step, we consider the periodic spline space on a uniform grid over 
$(-2,2)$ 
with grid size $h$, namely
\begin{equation}\label{eq:sper}
\Sper_{p,h}	:= \left\{	v_h \in S_{p,h}(-2,2) :	\tfrac{\partial^j}{\partial x^j} v_h(-2) = \tfrac{\partial^j}{\partial x^j}v_h(2) \, \; \mbox{for} \;\, j \in \{0,\ldots,p-1\}	\right\}, 
\end{equation}
where
$ S_{p,h}(-2,2) = \left\{ v_h \in C^{p-1}[-2,2]: v_h|_{(hi,h(i+1))} \in \pi_p \, \; \mbox{for} \; \, i \in \{-2n+1,\ldots,2n\}\\[.5em] \right\}$
is the corresponding spline space. The dimension of $\Sper_{p,h}$ is independent of the spline degree, since we have 
$\dim \Sper_{p,h} = 4n$.
We define the nodes 
$x_{-2n+1}, x_{-2n+2}, \ldots, x_{2n} $, 
with
\begin{equation}  \label{eq:nodes}
    x_i :=
    \begin{cases}
              ih & \mbox{if $p$ is odd}, \\
            (i-1/2)h & \mbox{if $p$ is even}.
    \end{cases}
\end{equation}
Note that
\begin{equation} \label{eq:nodes_symmetry} x_{i} = -x_{\delta_p - i} \end{equation}
holds, where
\begin{equation} \label{eq:delta}
    \delta_p :=
    \begin{cases}
              0 & \mbox{if $p$ is odd}, \\
            1 & \mbox{if $p$ is even}.
    \end{cases}
\end{equation}
We consider the Lagrange basis functions $L_i$, for 
$i=-2n+1, \ldots,2n$, 
associated to these nodes, see \cite[Lecture 4]{Schoenberg1973}. These functions satisfy
\begin{equation}\label{eq:lagrange-bas}
    L_j\in\Sper_{p,h}
    \qquad\mbox{with}\qquad
    L_j(x_i)=\delta_{i,j},
\end{equation}
see Figure~\ref{fig3} for a visualization.
Unlike B-splines, they are supported on the whole domain and they change sign.
\begin{figure}[h]
	\centering
	\includegraphics[width=.9\textwidth]{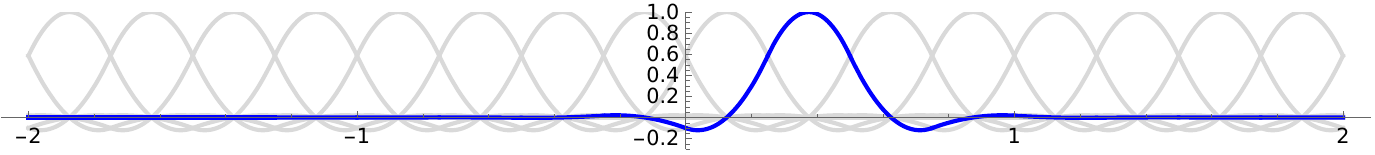}
	\caption{\label{fig3}Lagrange basis for periodic splines with degree $2$ over $(-2,2)$.}
\end{figure}

If we consider this basis, mass and stiffness matrices on $\Sper_{p,h}$ (as any matrix associated with a symmetric bilinear form) are symmetric and circulant. Therefore, their eigenvalues and eigenvectors are explicitly known. In particular, the eigenvectors have the form
\begin{equation} \label{eq:eigenvectors} \mathbf{v}^{(j)} = \begin{bmatrix} \omega^{ij} \end{bmatrix}_{i=-2n+1}^{2n} \in \mathbb{C}^{4n}, \qquad j = -2n+1,\ldots,2n, \end{equation}
where $\omega = e^{\mathbf{i} \frac{\pi}{2n}} = \cos \left( \frac{\pi}{2n} \right) + \mathbf{i} \sin \left( \frac{\pi}{2n} \right) $ is the complex $4n-$th root of unity. 
Moreover, if $\lambda_j$ denotes the eigenvalue associated to $\mathbf{v}^{(j)}$ for the considered matrix (mass or stiffness), since the latter is symmetric and circulant, it holds (see e.g. \cite{Tee2007,GolubVanLoan})
$$ \lambda_{j} = \lambda_{-j}, \qquad j=1,\ldots,2n-1. $$ 
As a consequence, any linear combination of $\mathbf{v}^{(j)}$ and $\mathbf{v}^{(-j)}$ is still an eigenvector.

We now select particular eigenvectors, based on the set $\Bdy$ that describes the boundary conditions imposed on the univariate spline space. In the following, we discuss in detail the case $\Bdy = \{ 0,1\}$, and then comment on the other cases.
For odd $p$, we consider the eigenvectors
$$ \mathbf{f}^{(j)}:=-\frac{\mathbf{i}}{2} \left( \mathbf{v}^{(2j)}- \mathbf{v}^{(-2j)} \right) = \begin{bmatrix} \sin \left( \frac{ij \pi }{n} \right)\end{bmatrix}_{i=-2n+1}^{2n}, $$
for $j=1,\ldots,n-1$. For even $p$, we instead consider the eigenvectors
$$ \mathbf{f}^{(j)} := -\frac{\mathbf{i}}{2} \left( \omega^{-j} \mathbf{v}^{(2j)} - \omega^{j} \mathbf{v}^{(-2j)} \right) = \begin{bmatrix} \sin \left( \frac{(2i-1) j  \pi }{2n} \right) \end{bmatrix}_{i=-2n+1}^{2n}, $$
for $j=1,\ldots,n-1$, and the additional eigenvector 
$$\mathbf{f}^{(n)} := \mathbf{v}^{(2n)} = \begin{bmatrix} \sin \left( (i-1/2)\pi \right) \end{bmatrix}_{i=-2n+1}^{2n} = \begin{bmatrix} (-1)^{(i-1)} \end{bmatrix}_{i=-2n+1}^{2n} . $$
In all of those cases, using~\eqref{eq:reg_space_dim}, \eqref{eq:nodes} and
\begin{equation}\label{eq:theta}
    \alpha_j := j \pi,
\end{equation}
we have  the representation
\[
    \mathbf{f}^{(j)} = \begin{bmatrix} \sin \left(   x_i \alpha_j \right)\end{bmatrix}_{i=-2n+1}^{2n}, \quad j=1,\dots,n_{reg},
\]
for the eigenvector and 
\begin{equation} \label{eq:discrete_fourier_fun}
F_j(x)= \sum_{i=-2n+1}^{2n} \sin(x_i\alpha_j) L_i(x), \qquad j = 1,\dots,n_{reg},
\end{equation}
for the corresponding eigenfunctions,
see Figure \ref{fig4} for a visualization. As eigenfunctions, they are in particular orthogonal with respect to the $L^2$ and $H^1$ scalar product.  

\begin{figure}[h]
	\centering
	\includegraphics[width=.9\textwidth]{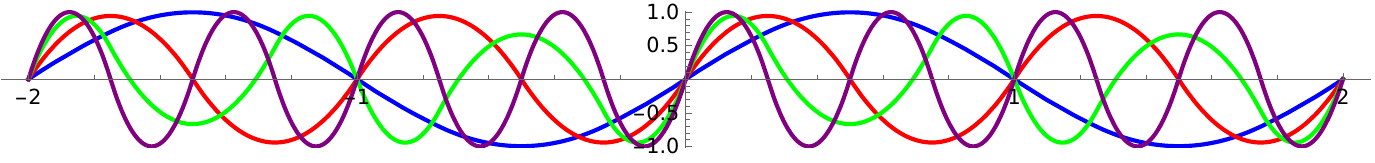}
	\caption{\label{fig4} Fourier modes with even symmetry on $\Bdy=\{0,1\}$, degree $2$ over $(-2,2)$.}
\end{figure}

We need the following lemma to proceed.
\begin{lemma} \label{lem:odd}
The discrete Fourier modes  \eqref{eq:discrete_fourier_fun} are 
odd symmetric with respect to $0$ and $1$, i.e. $F_j(-x)=-F_j(x)$ and $F_j(1-x)=-F_j(1+x)$.
\end{lemma}
\begin{proof}

We recall that a Lagrange basis function $L_i$ is even symmetric with respect to $x_i$, i.e. $L_i(x) = L_i(2x_i - x)$. Thus, it holds
\begin{equation} \label{eq:lagrange_property} L_{i}\left( -x \right) = L_{i}\left( 2 x_i + x\right) = L_i\left( h(2i-\delta_p) + x \right) = L_{\delta_p - i}(x), \end{equation}
for $i=-2n+1,\ldots,2n,$ where the last equality is a consequence of periodic Lagrange basis functions being translations of each other. 
We observe that, for odd $p$, in the right-hand side of \eqref{eq:discrete_fourier_fun} the addend corresponding to $i=n$ is null and therefore it does not contribute to the sum. Thus, we have
\begin{align*} F_{j}(-x) & = \sum_{i=-2n+1}^{2n} \sin(j \pi x_i) L_i(-x) = \sum_{i=-2n+1}^{2n} \sin(- j \pi x_{\delta_p - i}) L_{\delta_p - i}(x) \\ & = - \sum_{i=-2n+1}^{2n} \sin( j\pi x_{i}) L_{i}(x) = - F_j(x), \end{align*}
where in the second equality we used \eqref{eq:nodes_symmetry} and \eqref{eq:lagrange_property}, while in the last equality we simply rearranged the terms in the sum. 
The proof of $F_j(1-x)=-F_j(1+x)$ is analogous.
\end{proof}

Since the discrete Fourier functions are odd symmetric at $0$ and at $1$, we obtain
\[
			\partial^{2k} F_j(0)=\partial^{2k} F_j(1)=0 \quad\mbox{for all}\quad k\in \{0, \cdots, \lfloor \tfrac{p-1}2 \rfloor \}.
\]
The above equations show that the restrictions of the discrete Fourier functions to the interval $[0,1]$ are elements of the space $S_{p,h,\Bdy}^{reg}$ with $\Bdy=\{0,1\}$, which for the considered boundary conditions is defined in~\eqref{eq:reg_space} as
\[
    S_{p,h,\{0,1\}}^{reg}  = \{ v_h \in S_{p,h,\{0,1\}} : \partial^{2k} v_h(0)=\partial^{2k} v_h(1) = 0 \mbox{ for } k = 1,\ldots, \lfloor \tfrac{p-1}2 \rfloor \}.
\]
With a slight abuse of notation, we still denote these restrictions with $F_j$, $j=1,\ldots,n_{reg}$. Since they are linearly independent over $[0,1]$ and their number coincides with the dimension of $S_{p,h,\Bdy}^{reg}$, they form a basis for this space, see Figure~\ref{fig5} for a visualization. 

\begin{figure}[h]
	\centering
	\includegraphics[width=.5\textwidth]{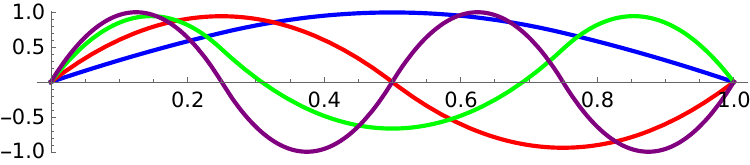}
	\caption{\label{fig5}Discrete Fourier functions as basis of $S_{p,h,\Bdy}^{reg}$ for $\Bdy=\{0,1\}$.}
\end{figure}

From \eqref{eq:discrete_fourier_fun} it immediately follows that
$$ F_{j}(x_i) = \sin(x_i\alpha_j), \qquad i,j=1,\ldots,n_{reg}, $$
with $x_i$ as in~\eqref{eq:nodes} and $\alpha_j$ as in~\eqref{eq:theta}.
Thus, each discrete Fourier mode $F_j$ is the function of $S_{p,h,\Bdy}^{reg}$ that interpolates the analytical eigenfunction of the Laplacian $\phi_j(x) = \sin(x\alpha_j)$ on the points $x_i$, $i=1,\ldots,n_{reg}$.
We remark that the uniqueness of the interpolant follows from \cite[Lecture 4, Theorem 4]{Schoenberg1973}.

We are now ready to prove the main theoretical result of this Section.

\begin{theorem} \label{thm:eigenfunctions}
The discrete Fourier functions form a basis of eigenfunctions for the mass and stiffness operators on the space $S_{p,h,\Bdy}^{reg}$.
\end{theorem}
\begin{proof}
Since the functions $F_j$, $j=1,\ldots,n_{reg}$, form a basis for $S_{p,h,\Bdy}^{reg}$, the statement is equivalent to showing that these functions are orthogonal with respect to the $L^2$ and $H^1$ inner products on the interval $[0,1]$. This follows from their orthogonality over $[-2,2]$ and Lemma \ref{lem:odd}. Indeed, since the product of odd symmetric functions is even symmetric, for every $i,j \in \left\{ 1,2, \ldots, n_{reg}\right\}, $ with $i \neq j$, we have 
		\[
				0 =	\int_{-2}^2 F_i(x)F_j(x) \mathrm dx
					=
					2\int_{0}^2 F_i(x)F_j(x) \mathrm dx
					=
					4\int_0^1 F_i(x)F_j(x) \mathrm dx.
		\]
The orthogonality for the derivatives can be proven analogously, since the derivative of an odd symmetric function is even symmetric.
\end{proof}

As a consequence of Theorem~\ref{thm:eigenfunctions}, if we fix a basis $\psi_1,\ldots,\psi_{n_{reg}}$ for $S_{p,h,\Bdy}^{reg}$, we can write the eigenvector matrix $Q_{reg}$ in \eqref{eq:eigendecomposition_reg} as
\begin{equation} \label{eq:reg_eigenvectors} 
    Q_{reg} = C^{-1}_{reg} U_{reg} D_{reg}^{-1},
\end{equation}
where $C_{reg},U_{reg},D_{reg} \in \mathbb{R}^{n_{reg} \times n_{reg}}$ are given by
$\left( C_{reg} \right)_{ij} = \psi_i \left( x_j \right) $ such that $C^{-1}_{reg}$ acts as an interpolation operator, $\left( U_{reg} \right)_{ij} = F_j(x_i) = \sin \left( x_i \alpha_j  \right) $ represents the Discrete Sine Transform and $D_{reg}$ is a diagonal matrix that normalizes the eigenvectors via
\begin{equation}\label{eq:Dreg}
\left( D_{reg} \right)_{ii} =  \|F_i\|_{L^2(0,1)}, \qquad i=1,\ldots,n_{reg}.
\end{equation}
Since the discrete Fourier modes $F_i$ are obtained by interpolating $\sin(x\alpha_i)$ by splines, $\|F_i\|_{L^2(0,1)}$ can be expressed using the mass matrix $M_{reg} = V_{reg}^T M V_{reg}$, the matrix $C^{-1}_{reg}$ representing the interpolation and the matrix $U_{reg}$ representing the Discrete Sine Transform:
\begin{equation}\label{eq:Dreg2}
\begin{aligned}
    D_{reg} &= \left( U_{reg}^T C^{-T}_{reg} V_{reg}^T M V_{reg} C^{-1}_{reg} U_{reg} \right)^{1/2}.
\end{aligned}
\end{equation}
By substituting~\eqref{eq:reg_eigenvectors} into~\eqref{eq:eigendecomposition_reg}, we obtain
\begin{equation}\label{eq:Lambda reg}
\begin{aligned}
   \Lambda_{reg} &= Q_{reg}^\top V_{reg}^\top K V_{reg} Q_{reg} = D_{reg}^{-1} U_{reg}^\top C^{-\top}_{reg} V_{reg}^\top K V_{reg} C^{-1}_{reg} U_{reg} D_{reg}^{-1}.
\end{aligned}
\end{equation}
As above, we have
\begin{equation}\label{eq:Lambda reg2}
\left( \Lambda_{reg} \right)_{ii} = |F_i|_{H^1(0,1)}^2\big/\|F_i\|_{L^2(0,1)}^2, \qquad i=1,\ldots,n_{reg}
.
\end{equation}
As already mentioned, we can assume that the basis functions of $S_{p,h,\Bdy}^{reg}$ have local support, which is equivalent to require that $C_{reg}$ is a banded matrix, and its bandwidth is roughly $p$. Therefore, assuming that its LU factors are computed in advance, applying $C_{reg}^{-1}$ to a vector requires $\mathcal O (n_{reg} p )$ flops. Moreover, the action of $U_{reg}$ on a vector corresponds to a Discrete Sine Transform (of type I or III, depending if $p$ is odd or even), a variant of the Discrete Fourier Transform which can be applied at $\mathcal O(n_{reg} \log n_{reg})$ cost. Thus, since we are assuming $n_{reg} \approx m$ where $m$ is the dimension of $S_{p,h,\Bdy}$, a single matrix-vector product with $Q_{reg}$ requires $\mathcal O \left( m \left( \log m + p \right) \right) $ flops. In Section~\ref{sec:special_basis}, we will see that for a particular choice of the basis it is possible to avoid the interpolation matrix $C_{reg}^{-1}$, further reducing the cost.

The matrices $D_{reg}$ and $\Lambda_{reg}$ can be computed using~\eqref{eq:Dreg}
and~\eqref{eq:Lambda reg}. The complexity of this computation, exploiting the sparsity of the involved matrices, is $\mathcal O (n_{reg}^2 p ) = \mathcal O (m^2 p ) $, which is typically not an issue when dealing with a $d-$dimensional problems, for $d \geq 2$. 
Alternatively, they can be computed with the following formulas (see also \cite{Lamsahel2025}).
\begin{lemma}\label{lem:coefs}
    For $j=1,\dots,n_{reg}$, 
    \begin{align*}
        (D_{reg})_{jj}
        &= h^{1/2} \frac{\big(\sum_{k=-p}^p \widetilde{\mathcal B}_{2p+1}(k) \cos (k h\alpha_j)\big)^{1/2}}
                    {\sum_{k=-\lfloor p/2\rfloor}^{\lfloor p/2\rfloor} \widetilde{\mathcal B}_{p}(k) \cos(k h \alpha_j)},
        \\
        (\Lambda_{reg})_{jj}
        &= \frac{2-2\cos(h\alpha_j)}{h}\, \frac{\sum_{k=-(p-1)}^{p-1} \widetilde{\mathcal B}_{2p-1}(k) \cos (k h\alpha_j)}{\sum_{k=-p}^p \widetilde{\mathcal B}_{2p+1}(k) \cos (k h\alpha_j)}
    \end{align*}
    holds, where $\widetilde{\mathcal B}_p$ is the symmetric cardinal B-spline of degree $p$, given by the recurrence relation
    \begin{equation}\label{eq:cardinal}
        \widetilde{\mathcal B}_p(x) =  \frac{p+1+2x}{2p} \widetilde{\mathcal B}_{p-1}(x+\tfrac12) + \frac{p+1-2x}{2p} \widetilde{\mathcal B}_{p-1}(x-\tfrac12),\quad
        \widetilde{\mathcal B}_0(x) = \begin{cases} 1 & \mbox{if}\; x\in[-\tfrac12,\tfrac12), \\ 0 & \mbox{otherwise.}\end{cases}
    \end{equation}
\end{lemma}
A proof of this Lemma is given in the Appendix.

\paragraph{Other boundary conditions.}
Up to this point, we only considered the case $\Bdy= \{ 0, 1\}$. However, similar results hold when $\Bdy \neq \{ 0,1 \}$. More precisely, it can be shown that the discrete eigenfunctions of the Laplacian on $S_{p,h,\Bdy}^{reg}$ interpolate the first $n_{reg}$ analytical ones, given by
\[
        \phi_j (x) =\sin (\alpha_j x+\beta),
\]
where the values of $\alpha_j$ and $\beta$ are as given in Table~\ref{tab:all_cases}.
\begin{figure}[h]
	\centering
	\includegraphics[width=.9\textwidth]{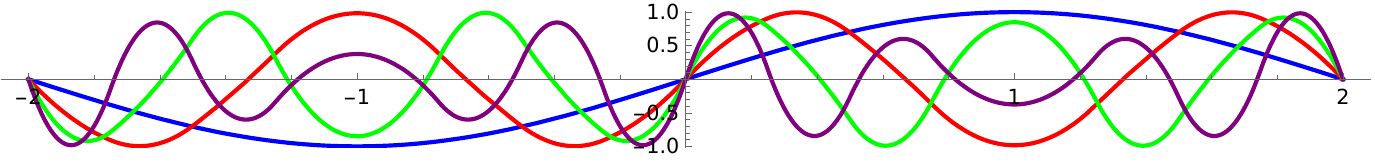}
	\includegraphics[width=.9\textwidth]{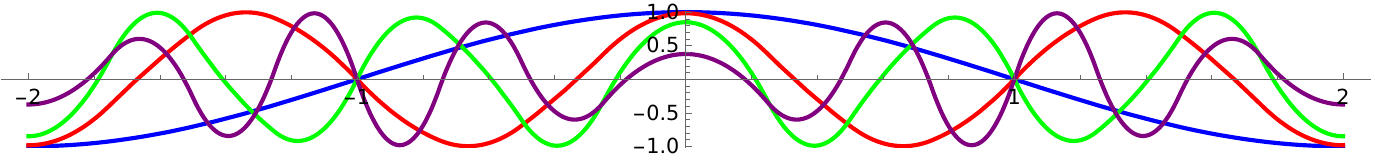}
	\includegraphics[width=.9\textwidth]{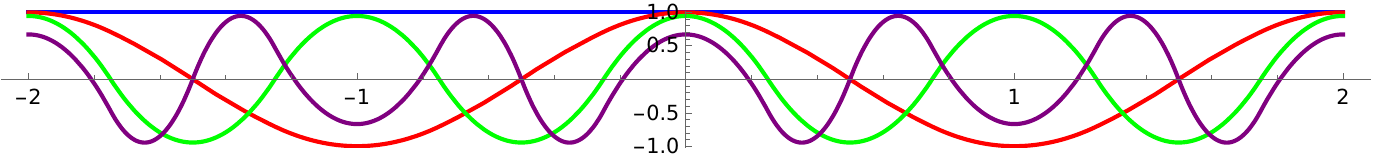}
	\caption{\label{fig:pic4a} Fourier modes with even symmetry on $\Bdy$ and odd symmetry on $\{0,1\}\setminus \Bdy$ and degree $2$ for the choices $\Bdy=\{0\}$ (top), $\Bdy=\{1\}$ (middle) and $\Bdy=\emptyset$ (bottom).}
\end{figure}
The expression of the discrete eigenfunctions can be found as follows.
Once again, the eigenvectors of the mass and stiffness matrices for $\Sper_{p,h}$ with respect to the Lagrangian basis are considered. By taking linear combinations of vectors \eqref{eq:eigenvectors}, for each $j=1,\ldots,n_{reg}$ one can construct an eigenvector $\mathbf{f}^{(j)}$ by imposing $ \mathbf{f}^{(j)}_i =\phi_j(x_i)$ for $i=-2n+1, \ldots, 2n$. 
It can then be shown that the modes corresponding to these eigenvectors are odd symmetric with respect to $x$ 
if $x \in \Bdy$ and even symmetric with respect to $x$ 
if $x \in \{ 0,1 \} \setminus \Bdy$, cf. Figure~\ref{fig:pic4a}.
One can use this fact to prove that the orthogonality of these discrete eigenfunctions in the $L^2$ and $H^1$ inner products over $[-2,2]$  still holds if these functions are restricted to $[0,1]$. 
As a final step, it can be shown that these restrictions form a basis for $S_{p,h,\Bdy}^{reg}$, which implies that they are the discrete eigenfunctions of the mass and stiffness operators on this space.

{\renewcommand\arraystretch{1.5} 
\begin{table}[htp]
\begin{center}
\fontsize{9}{11}\selectfont
\begin{tabular}{|c|c|c|c|c|}
\hline
\multicolumn{5}{|c|}{  degree $p$ is odd } \\
\hline
& $\Bdy=\{ 0, 1 \}$ & $\Bdy=\{ 0 \}$ & $\Bdy=\{ 1 \}$ & $\Bdy=\emptyset$ \\
\hline
$n_{reg}$ &
$n-1$ &
$n$ &
$n$ &
$n+1$ \\[1mm]
\hline
$x_i$ &
$ih$ &
$ih$ &
$(i-1)h$ &
$(i-1)h$ \\[1mm]
\hline
$\alpha_j$ &
$j\pi$ &
$(j-1/2)\pi$ &
$(j-1/2)\pi$ &
$(j-1)\pi$ \\[1mm]
\hline
$\beta$ &
$0$ &
$0$ &
$\pi/2$ &
$\pi/2$ \\[1mm]
\hline
$\left( U_{reg} \right)_{ij}$ 
& 
$\sin \left( \tfrac{ij \pi}{n} \right)$ &
$\sin \left( \tfrac{i (2j-1) \pi}{2n} \right)$ & 
$\cos \left( \tfrac{(i-1)(2j-1) \pi}{2n} \right)$ &
$\cos \left( \tfrac{(i-1)(j-1) \pi}{n} \right)$ \\[1mm]
\hline
{\renewcommand\arraystretch{1.2}  \begin{tabular}{c} Trigonometric \\ Transform \end{tabular} }& DST-I & DST-II & DCT-II & DCT-I \\
\hline
{\renewcommand\arraystretch{1.2}  \begin{tabular}{c} Adjoint \\ Transform \end{tabular} }& DST-I & DST-III & DCT-III & DCT-I \\
\hline
\hline
\multicolumn{5}{|c|}{   degree $p$ is even } \\
\hline
& $\Bdy=\{ 0, 1 \}$ & $\Bdy=\{ 0 \}$ & $\Bdy=\{ 1 \}$ & $\Bdy=\emptyset$ \\
\hline
$n_{reg}$ &
$n$ &
$n$ &
$n$ &
$n$ \\[1mm]
\hline
$x_i$ &
$(i-1/2)h$ &
$(i-1/2)h$ &
$(i-1/2)h$ &
$(i-1/2)h$ \\[1mm]
\hline
$\alpha_j$ &
$j\pi$ &
$(j-1/2)\pi$ &
$(j-1/2)\pi$ &
$(j-1)\pi$ \\[1mm]
\hline
$\beta$ &
$0$ &
$0$ &
$\pi/2$ &
$\pi/2$ \\[1mm]
\hline
$\left( U_{reg} \right)_{ij}$ & 
$\sin \left( \tfrac{\left(2i-1\right)j \pi}{2n} \right)$ &
$\sin \left( \tfrac{(2i-1) (2j-1) \pi}{4n} \right)$ &
$\cos \left( \tfrac{(2i-1)(2j-1) \pi}{4n} \right)$ &
$\cos \left( \tfrac{(2i-1)(j-1) \pi}{2n} \right)$  \\[1mm]
\hline
{\renewcommand\arraystretch{1.2}  \begin{tabular}{c} Trigonometric \\ Transform \end{tabular} }& DST-III & DST-IV & DCT-IV & DCT-III \\
\hline
{\renewcommand\arraystretch{1.2}  \begin{tabular}{c} Adjoint \\ Transform \end{tabular} }& DST-II & DST-IV & DCT-IV & DCT-II \\
\hline
\end{tabular}

\caption{Entries of matrix $U_{reg}=\big[ \sin( x_i \alpha_j +\beta) \big]_{i,j=1}^{n_{reg}}$ and trigonometric transforms associated to $U_{reg}$ and $U_{reg}^\top$, depending on parity of $p$ and Dirichlet boundary $\Bdy$.
}
\label{tab:all_cases}
\end{center}
\end{table}}

In conclusion, 
the eigenvector matrix $Q_{reg}$ has an expression analogous to \eqref{eq:reg_eigenvectors},
where matrix $U_{reg}$ contains the analytical eigenfunctions evaluated at the nodes $x_i$ that belong to $[0,1]\setminus \Bdy$; the number of those nodes is $n_{reg}$. Matrix-vector products with $U_{reg}$ can be computed using a particular variant of the Discrete Fourier Transform. In total, there are eight possible cases: 
two (one for even $p$ and one for odd $p$) for each of the four possible boundary configurations encoded by the set $\Bdy$. These eight cases correspond to as many trigonometric transforms, namely the Discrete Sine and Cosine Transforms (abbreviated with DST and DCT, respectively) of types I-IV, see e.g. \cite{Strang1999,Britanak2007} for more details.
We collect all cases in Table \ref{tab:all_cases}, where for each case we report the entries of the matrix $U_{reg}$, as well as the corresponding trigonometric transform. For completeness, in each case we also report the name of the ``adjoint'' transform, i.e. the transform associated with $U_{reg}^\top$, which is also required for the application of the preconditioner \eqref{eq:preconditioner}.

\section{Expression of the eigenvectors for a special basis}
\label{sec:special_basis}

Here we specialize the results of Section~\ref{sec:discrete_eigenfunctions} to a specific choice of the basis for $S_{p,h,\Bdy}^{reg}$, and show that in this case there is no need for an interpolation matrix as in~\eqref{eq:reg_eigenvectors}, i.e. each eigenvector is, up to scaling, the evaluation of a sinusoidal function on equally spaced points. This is in agreement with the results of \cite{Deng2023,Lamsahel2025}. 

For the sake of completeness, here we report the details on the constructions of the basis, assuming 
for simplicity $n > p$ (see e.g. \cite{Lamsahel2025}  for the general case).
The cardinal B-splines (see e.g. \cite[Chapter XVI]{Boor2001}) of degree $p$ that are active on $(0,1)$ are given by
\begin{equation} \label{eq:Btilde}
    \widetilde B_i(x) := \widetilde{\mathcal B}_p( x/h - x_i), \qquad i=-
    \left\lfloor \frac{p-1}{2} \right\rfloor, 
    -\left\lfloor \frac{p-1}{2} \right\rfloor+1, \ldots, n + \left\lfloor \frac{p}{2} \right\rfloor,
\end{equation}
with $\widetilde{\mathcal{B}}_p$ as in~\eqref{eq:cardinal} and $x_i$ as in~\eqref{eq:nodes}.
This somewhat unusual numbering guarantees that each $\widetilde{B}_i$ takes its maximum at $x_i$.

We consider the ordered set
\[ \mathcal{B}_L = \begin{cases}
\left\{ \widetilde{B}_i - \widetilde{B}_{\delta_p - i} \right\}_{i=1}^{\lfloor p/2 \rfloor} &\mbox{if } 0 \in \Bdy, \\[2mm]
\left\{\widetilde{B}_i + \widetilde{B}_{1 - i} \right\}_{i=1}^{ p/2 } &\mbox{if } 0 \notin \Bdy, p \mbox{ even}, \\[2mm]
\left\{\widetilde{B}_0 \right\} \cup \left\{\widetilde{B}_i + \widetilde{B}_{-i} \right\}_{i=1}^{\lfloor (p-1)/2 \rfloor} &\mbox{if } 0 \notin \Bdy, p \mbox{ odd},	\end{cases} \]
where here and in the rest of this section the symbol ``$\cup$'' denotes the ordered union of ordered sets, and $\delta_p$ is defined in \eqref{eq:delta}.
Analogously, we consider
\[ \mathcal{B}_R = \begin{cases}
\left\{ \widetilde{B}_i - \widetilde{B}_{2n + \delta_p - i} \right\}_{i = n - \lfloor (p-1)/2 \rfloor}^{n-1+\delta_p} &\mbox{if } 1 \in \Bdy, \\[2mm]
\left\{\widetilde{B}_i + \widetilde{B}_{2n + 1 - i} \right\}_{i= n - p/2 - 1}^{ n } &\mbox{if } 1 \notin \Bdy, p \mbox{ even}, \\[2mm]
\left\{\widetilde{B}_i + \widetilde{B}_{2n - i} \right\}_{i=n - (p-1)/2}^{n-1} \cup \left\{ \widetilde{B}_n  \right\} &\mbox{if } 1 \notin \Bdy, p \mbox{ odd},	\end{cases} \]
and
\[  \mathcal{B}_I = \left\{ \widetilde{B}_i \right\}_{i=\lfloor p/2 \rfloor + 1}^{n - \lfloor (p+1)/2 \rfloor}.  \]
Finally, we define
\begin{equation} \label{eq:special_basis}
\mathcal{B} := \left\{ \widehat{B}_i \right\}_{i=1}^{n_{reg}} := \mathcal{B}_L \cup \mathcal{B}_I \cup \mathcal{B}_R.
\end{equation}
This set of functions, restricted to $[0,1]$, is a basis for $S^{reg}_{p,h,\Bdy}$, see~\cite{Takacs2016} for the case $\Bdy=\emptyset$. Moreover, the following result holds.

\begin{theorem} \label{thm:special_basis}
Let $C_{reg} \in \mathbb{R}^{n_{reg} \times n_{reg}}$ with $\left( C_{reg} \right)_{ij} = \widehat{B}_j(x_i) $, for $i,j = 1,\ldots,n_{reg}$, then the columns of $U_{reg}$, as defined in Table \ref{tab:all_cases}, are eigenvectors for $C_{reg}$.
Precisely, it holds
\[ C_{reg} U_{reg} = U_{reg} \Theta  \]
where $\Theta$ is a diagonal matrix with $ \displaystyle \left(\Theta\right)_{jj} = \sum_{k=-\lfloor p/2 \rfloor}^{\lfloor p/2 \rfloor} \widetilde{\mathcal B}_p(k) \cos \left( k h \alpha_j \right)$, for $j=1,\ldots,n_{reg}$. 
\end{theorem}
\begin{proof}

We give details of the proof only for the case $\Bdy = \{0,1\}$, as in the other cases the proof can be obtained analogously. When $\Bdy = \{0,1\}$, the basis functions have the form
\begin{equation} \label{eq:special_basis_dirichlet} \widehat{B}_i =	\begin{cases} \widetilde{B}_i - \widetilde{B}_{\delta_p - i} &\mbox{for } i=1,\ldots,\lfloor p/2 \rfloor,\\ \widetilde{B}_i & \mbox{for } i=\lfloor p/2 \rfloor + 1, \ldots, n_{reg} - \lfloor p/2 \rfloor, \\ \widetilde{B}_i-\widetilde{B}_{n + n_{reg} + 1 - i} & \mbox{for } i=  n_{reg} - \lfloor p/2 \rfloor + 1,\ldots, n_{reg}. \end{cases} \end{equation}
We preliminary observe that, since $x_{\delta_p - k} = - x_{k}$ for $k=1,\ldots,\lfloor p/2 \rfloor$, 
it holds
\begin{equation} \label{eq:identity1}
-\sum_{k=1}^{\lfloor p/2 \rfloor} \widetilde{B}_{\delta_p - k}(x) \sin(x_i\alpha_j) = \sum_{k=1}^{\lfloor p/2 \rfloor} \widetilde{B}_{\delta_p - k}(x) \sin(x_{\delta_p - k}\alpha_j) = \sum_{k=-\lfloor (p-1)/2 \rfloor}^{0} \widetilde{B}_{k}(x) \sin(x_{k}\alpha_j)
\end{equation}
for $x \in [-ph, \, 1+ph]$ and for $j=1,\ldots,n_{reg}$. Note that, for odd $p$, in the last sum we added the null term corresponding to $j=0$. Similarly, it holds
\begin{equation} \label{eq:identity2}
-\sum_{k=n_{reg} - \lfloor p/2 \rfloor + 1}^{n_{reg}} \widetilde{B}_{n + n_{reg} + 1 - i}(x) \sin(x_k\alpha_j) = \sum_{k=n_{reg} + 1}^{n + \lfloor p/2 \rfloor} \widetilde{B}_{k}(x) \sin(x_k\alpha_j).
\end{equation}
Given $j \in \{ 1,\ldots,n_{reg} \}$, we denote with 
\[ \mathbf{u}^{(j)} = \left[ \sin\left( x_i \alpha_j \right) \right]_{i=1}^{n_{reg}} \] 
the $j-$th column of $U_{reg}$. For $i = 1,\ldots,n_{reg},$ it holds
\begin{align} \label{eq:step1}
\left( C_{reg} \mathbf{u}^{(j)}  \right)_i & = \sum_{k=1}^{n_{reg}} \widehat{B}_k(x_i) \sin \left( x_k\alpha_j \right) = \sum_{k=-\lfloor (p-1)/2 \rfloor}^{n + \lfloor p/2 \rfloor} \widetilde{B}_k(x_i) \sin \left( x_k\alpha_j \right) \\
\label{eq:step2} & = \sum_{k=i-\lfloor p/2 \rfloor}^{i + \lfloor p/2 \rfloor} \widetilde{B}_k(x_i) \sin \left( x_k\alpha_j \right) = \sum_{k=-\lfloor p/2 \rfloor}^{\lfloor p/2 \rfloor} \widetilde{\mathcal B}_p(k) \sin \left(  \left( x_i + kh \right) \alpha_j \right) \\
\label{eq:step3} & = \sum_{k=-\lfloor p/2 \rfloor}^{\lfloor p/2 \rfloor} \widetilde{\mathcal B}_p(k) \left( \cos \left( k h \alpha_j \right) \sin \left( x_i \alpha_j \right) + \cos \left(  x_i \alpha_j \right) \sin \left(  k h \alpha_j  \right) \right),
\end{align}
where the last equality of \eqref{eq:step1} follows from the definition of the basis functions \eqref{eq:special_basis_dirichlet} and from  \eqref{eq:identity1} and \eqref{eq:identity2}, while the first equality of \eqref{eq:step2} is obtained simply by neglecting the indices $k$ such that $\widetilde{B}_{k}(x_i) = 0$ in the summation. 
The second equality \eqref{eq:step2} follows from \eqref{eq:Btilde}, while
\eqref{eq:step3} follows from the sum formula for the sine function.
We now observe that
\[ 
\sum_{k=-\lfloor p/2 \rfloor}^{\lfloor p/2 \rfloor} \widetilde{\mathcal B}_p(k)  \cos \left(  x_i \alpha_j \right) \sin \left(  k h  \alpha_j \right) = 0
\]
since the addends are odd with respect to the summation index $k$.
In conclusion, we have
\[ \left( C_{reg} \mathbf{u}^{(j)}  \right)_i = \sum_{k=-\lfloor p/2 \rfloor}^{\lfloor p/2 \rfloor} \widetilde{\mathcal B}_p(k) \cos \left( k h \alpha_j \right) \sin \left( x_i \alpha_j \right) = 
\left( \Theta \right)_{jj} \left( \mathbf{u}^{(j)} \right)_i.
\]
\end{proof}

Thanks to this theorem, we can specialize the structure of the eigenvector matrix \eqref{eq:reg_eigenvectors} when the basis \eqref{eq:special_basis} is considered. Indeed, in this case we have
\[ 
Q_{reg} = U_{reg} \widetilde{D}_{reg}^{-1},
\]
where $\widetilde{D}_{reg} =  D_{reg} \Theta$ is a diagonal matrix. Equivalently, if the mass and stiffness matrices of $S^{reg}_{p,h,\Bdy}$ are written with respect to the basis \eqref{eq:special_basis}, then they are diagonalized by $U_{reg}$.
We emphasize that a single matrix-vector product with $Q_{reg}$ now requires only $\mathcal O \left( m \log m \right) $ flops. 
We also emphasize that $\widetilde{D}_{reg}$ can be computed either by combining the expressions of $D_{reg}$ and $\Theta$ given by Lemma \ref{lem:coefs} and Theorem \ref{thm:special_basis}, or directly as $ \widetilde{D}_{reg} = \left( U_{reg}^\top V_{reg}^\top M V_{reg} U_{reg} \right)^{1/2}.$

\section{Numerical results}
\label{sec:num_res}
\newcommand{\z}{\phantom{0}}

In this Section we numerically assess the performance of the proposed preconditioning strategy. 
All experiments are performed using Matlab 2024b with the GeoPDEs 3.0 toolbox \cite{Vazquez2016}. 
In each test, we discretize the Poisson problem using the same number of elements $n$ and the same spline degree $p$ in every parametric direction. We consider homogeneous Dirichlet boundary conditions everywhere (meaning that $\Bdy =\{ 0,1\}$ for the univariate spaces) and globally $C^{p-1}$ splines, unless otherwise stated. As right-hand sides, we consider random vectors.
The resulting linear systems are solved using the Conjugate Gradient method, preconditioned with the IFFD method, with the zero vector as initial guess and tolerance on the relative residual set equal to $10^{-8}$.  
The eigendecomposition \eqref{eq:eigendecomposition_out} on $S_{p,h,\Bdy}^{out}$ is computed numerically using the Matlab function {\tt eig}.

\paragraph{Square domain.}
In the first experiment, we consider the unit square $\Omega = [0,1]^2$. Since there is no geometry, the system matrix has the structure \eqref{eq:system_matrix}. Moreover, here $K_1 = K_2 $ and $M_1 = M_2 $, therefore $Q_1=Q_2 =:Q$. The purpose of this test is to assess the quality of the approximation $\widetilde{Q} \approx Q$. Indeed, if this approximation were exact, the method would converge in 1 iteration. 
The number of PCG iterations is reported in Table \ref{tab:square}. These results show that the IFFD preconditioner is clearly robust in both $h$ and $p$, as predicted by Theorem \ref{thm:stable_splitting}, and the number of iterations is very low.
Note that for $p=2$ the IFFD preconditioner is actually a direct solver, as predicted by the theory. Indeed, as observed in Section \ref{sec:splitting}, in this case we have for all directions $S_{p,h,\Bdy} = S_{p,h,\Bdy}^{reg}$ and therefore the whole mass and stiffness matrices are diagonalized by $\widetilde{Q}$. 

 {\renewcommand\arraystretch{1.4} 
\begin{table}[ht]
\begin{center}
\begin{tabular}{|c|c|c|c|c|c|c|}
\hline
$n$ & $p=2$ &$p=3$ & $p=4$ & $p=5$ & $p=6$ & $p=7$  \\
 \hline
 128 & 1 &  7  &  6  &  6  &  6  &  6 \\
\hline
256 & 1 &  7  &  6  &  6  &  6  &  6 \\
\hline
512 & 1 &  7  &  6  &  6  &  6  &  6 \\
\hline
\end{tabular}
\caption{Square domain. Number of PCG iterations.}
\label{tab:square}
\end{center}
\end{table}}

\paragraph{Cube domain with mixed boundary conditions.}
We now assess the quality of the approximation of the eigenvectors when different boundary conditions are considered, that is $\Bdy \neq \{ 0, 1\}$ in the univariate spaces. For this purpose, we again consider a domain with trivial geometry, i.e. the unit cube $\Omega = [0,1]^3$. We consider Dirichlet conditions on the boundaries corresponding to $x=0$ and $y=1$, and Neumann conditions on the remaining boundaries. 
We emphasize that the chosen problem yields in the univariate spaces all the combinations of boundary conditions other than $\Bdy = \{ 0, 1\}$. 
Table \eqref{tab:cube} collects the number of PCG iterations obtained for this problem, which is analogous to the previous case. The only exception is the case $p=2$, since here the univariate outlier spaces are not trivial, as discussed in Section \ref{sec:splitting}, and therefore the diagonalization is not exact. 

{\renewcommand\arraystretch{1.4} 
\begin{table}[ht]
\begin{center} 
\begin{tabular}{|c|c|c|c|c|}
\hline
$ n $ & $p = 2$ & $p = 3$ & $p=4$ & $p=5$ \\
\hline
\z64 &  7  &  7  &  7  &  6  \\
\hline
128 &  7  &  7  &  6  &  6  \\
\hline
256 &  7  &  6  &  6  &  6  \\
\hline
\end{tabular} 
\caption{Cube domain. Number of PCG iterations.}
\label{tab:cube}
\end{center} \end{table}
}

\paragraph{Thick quarter of annulus domain.}
We now consider a three-dimensional problem with non-trivial geometry, i.e. the thick quarter of annulus shown in Figure \ref{fig:thick_quater_of_annulus}.
We consider Dirichlet boundary conditions on the bottom face, and Neumann boundary conditions on the other faces. We compare the performance of the IFFD preconditioner \eqref{eq:preconditioner}
with the standard FD preconditioner \eqref{eq:exact_diagonalization}, obtained by numerically computing the exact eigendecompositions \eqref{eq:eigendecomposition}. The results are reported in Table \ref{tab:thick_quarter_of_annulus}, and we see that the number of PCG iterations is almost identical for the two approaches. We recall that, on a three-dimensional problem, each application of \eqref{eq:exact_diagonalization} requires $\mathcal O(n^4)= \mathcal O \left( N_{dof}^{4/3} \right)$ flops, where $N_{dof}$ is the number of degrees of freedom, while \eqref{eq:preconditioner} only requires $ \mathcal O \left( N_{dof} \left( \log N_{dof} + p \right) \right) $ flops.

\begin{figure}[ht]
	\centering
	\includegraphics[width=.5\textwidth]{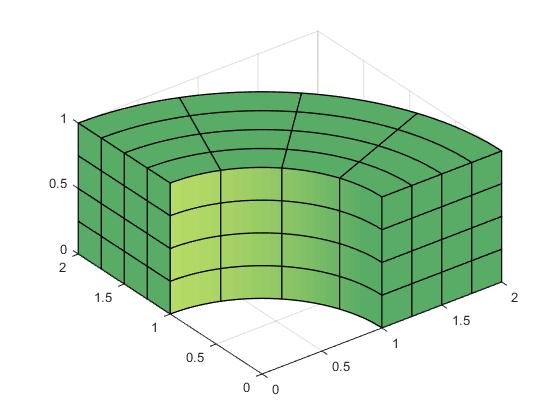}
	\caption{\label{fig:thick_quater_of_annulus} Thick quarter of annulus domain.}
\end{figure}

{\renewcommand\arraystretch{1.4} 
\begin{table}[ht]
\begin{center} 
\begin{tabular}{|c|c|c|c|c|}
\hline
$ n $ & $p = 2$ & $p = 3$ & $p=4$ & $p=5$ \\
\hline
16 & 28 / 29 & 28 / 29 & 28 / 29 & 29 / 30 \\
\hline
32 & 28 / 30 & 28 / 29 & 29 / 29 & 29 / 30 \\
\hline
64 & 28 / 30 & 28 / 30 & 29 / 30 & 29 / 30 \\
\hline
\end{tabular} 
\caption{ Thick quarter of annulus domain. 
Number of PCG iterations obtained with the FD preconditioner (left numbers) and with the IFFD preconditioner (right numbers).
}
\label{tab:thick_quarter_of_annulus}
\end{center} \end{table}
}

\paragraph{Plate with hole domain parametrized with a $C^0$ line.}
In this test we address a problem where the spline space is not globally $C^{p-1}$. For this purpose, we consider the plate with hole domain shown in Figure \ref{fig:plate_with_hole}. The parametrization, which is taken from the GeoPDEs toolbox, is $C^0$ along the line connecting the upper-left corner with the bottom-left quarter of a circle. Thus, it is natural to consider a spline space having the same regularity. The IFFD preconditioner is then constructed following Remark \ref{rmk:reduced_regularity}. 
The number of PCG iterations, reported in Table \ref{tab:plate_with_hole}, indicates that our preconditioning strategy is robust with respect to $h$ and $p$ also in this case. For completeness, as in the previous case we also report the number of PCG iterations obtained with the standard FD method, which are again almost identical.

\begin{figure}[ht]
	\centering
	\includegraphics[width=.5\textwidth]{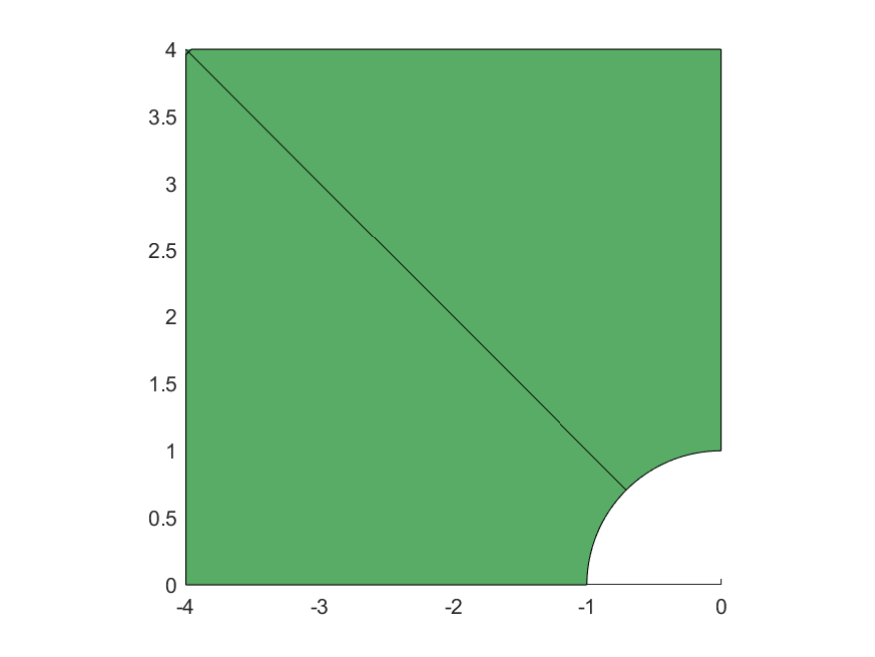}
	\caption{\label{fig:plate_with_hole} Plate with hole domain.}
\end{figure}

{\renewcommand\arraystretch{1.4} 
\begin{table}[ht]
\begin{center} 
\begin{tabular}{|c|c|c|c|c|}
\hline
$ n $ & $p = 2$ & $p = 3$ & $p=4$ & $p=5$ \\
\hline
\z64 & 27 / 27  & 27 / 29  & 28 / 28  & 29 / 29  \\
\hline
128 & 28 / 28  & 28 / 30  & 29 / 29 & 29 / 29 \\
\hline
256 & 28 / 29  & 28 / 30  & 29 / 30 & 29 / 30  \\
\hline
\end{tabular} 
\caption{Plate with hole domain. 
Number of PCG iterations obtained with the FD preconditioner (left numbers) and with the IFFD preconditioner (right numbers).
}
\label{tab:plate_with_hole}
\end{center} \end{table}
}

\section{Conclusions}
\label{sec:conclusions}

We proposed a preconditioner for the isogeometric discretization of the Poisson problem, based on an approximated version of the Fast Diagonalization method.
This approximation is obtained by splitting each univariate spline space in two parts, and then by diagonalizing the problem separately on each subspace. In particular, on the larger subspace, this diagonalization can be performed by leveraging the Fast Fourier Transform. We showed that the resulting preconditioner is spectrally equivalent to the standard Fast Diagonalization preconditioner (hence it is robust with respect to the mesh size $h$ and the degree $p$) and yields an application cost proportional to $N_{dof} \left( \log N_{dof} + p \right)$, where $N_{dof}$ is the number of degrees of freedom. This approach is referred to as Isogeometric Fast Fourier-based Diagonalization (IFFD) method.

The proposed method is applicable in cases where the overall computational domain is parameterised using a single geometry function (single-patch case). In more complicated problems, the overall domain might be represented as a collection of multiple, separately parameterized patches (multi-patch case). In order to solve linear systems obtained in the multi-patch setting, domain decomposition solvers, like overlapping Schwarz \cite{doi:10.1137/110833476}, Finite Element / IsogEometric Tearing and Interconnecting (FETI/IETI) \cite{KLEISS2012201}, or Balancing Domain Decomposition by Constraints (BDDC) \cite{doi:10.1142/S0218202513500048} methods have been proposed.

In order to realize these methods, linear systems that are local to one or a few patches have to be solved. This can be done using direct solvers for small subproblems, but alternatives are required when these subproblems are large. The approach proposed in this paper can be used, provided that the local problems exhibit an underlying tensor structure analogous to single-patch problems. This can be achieved by properly choosing (or adjusting) the domain decomposition method to be used, as it is done, for example, in \cite{BOSY20202604} for a FETI approach, or in \cite{doi:10.1142/S0218202523500495} for a Discontinuous Galerkin approach. In these works the Fast Diagonalization method is considered as an inexact solver for the local problems, and it could be replaced seamlessly with the IFFD method.

\section*{Appendix}

In order to show Theorem~\ref{thm:stable_splitting}, we use the following Lemma.
\begin{lemma}\label{lem:approxerror}
    There is a projector $\Pi_{p,h,\Bdy}: S_{p,h,\Bdy} \to S_{p,h,\Bdy}^{reg}$ such that
    \begin{align}\label{eq:lem:approxerror1}
    |\Pi_{p,h,\Bdy}v|_{H^1(0,1)}
    &\le
    |v|_{H^1(0,1)},
    \\
    \label{eq:lem:approxerror2}
    \| v - \Pi_{p,h,\Bdy} \, v \|_{L^2(0,1)}
    &\le \frac{h}{\pi} |v|_{H^1(0,1)}  
    \end{align}
    for all $v\in H^1_\Bdy(0,1)$.
\end{lemma}
\begin{proof}
    Let
    \[
        (u,v)_{H^1_\Bdy(0,1)}
        :=
        \begin{cases}
            (u', v')_{L^2(0,1)} & \mbox{if}\;\, \Bdy\ne \emptyset,\\
            (u', v')_{L^2(0,1)} + (u,1)_{L^2(0,1)}(v,1)_{L^2(0,1)} & \mbox{if}\;\, \Bdy = \emptyset
        \end{cases}
    \]
    be a scalar product on $H^1_\Bdy(0,1)$. We choose $\Pi_{p,h,\Bdy}$ to be the orthogonal projector $H^1_\Bdy(0,1)\to S_{p,h,\Bdy}$ with respect to this scalar product.
    For $\Bdy\ne\emptyset$, this projector is the standard $H^1$ projector, so~\eqref{eq:lem:approxerror1} is obvious. For $\Bdy=\emptyset$, we use that $\Pi_{p,h,\Bdy}c = c$ for all $c\in\mathbb R$ to obtain $|\Pi_{p,h,\Bdy}v|_{H^1(0,1)}^2 + (v,1)_{L^2(0,1)}^2 = |\Pi_{p,h,\Bdy}v|_{H^1_\Bdy(0,1)}^2 \le |v|_{H^1_{\Bdy}(0,1)}^2 = |v|_{H^1(0,1)}^2 + (v,1)_{L^2(0,1)}^2$ from which~\eqref{eq:lem:approxerror1} immediately follows.
        
    The proof of~\eqref{eq:lem:approxerror2} is analogous to the proof of Theorem~1.1 in~\cite{Takacs2016}, where only the case with pure Neumann conditions ($\Bdy=\emptyset$) was considered. 
    First, define an extension operator $E_\Bdy: H^1_\Bdy(0,1) \to H^1_{per}(-2,2):= \{ v\in H^1(-2,2) : v(-2)=v(2)\}$, defined via
    \begin{equation}\label{eq:wdef}
            (E_\Bdy v)(x):=
            \begin{cases}
                    (-1)^{\Bdy_0} (-1)^{\Bdy_1} v(x-2) & \mbox{ if } x\in [-2,-1), \\
                    (-1)^{\Bdy_0} v(-x) & \mbox{ if } x\in [-1,0),\\
                    v(x) &  \mbox{ if } x\in [0,1), \\
                    (-1)^{\Bdy_1} v(2-x) & \mbox{ if } x\in [1,2], \\
            \end{cases}
    \end{equation}
    where $\Bdy_x$ is $1$ if $x\in \Bdy$ and $0$ otherwise.
    By construction, we have $|E_\Bdy v|_{H^1(-2,2)}=2|v|_{H^1(0,1)}$ for all $v\in H^1_\Bdy(0,1)$.
    Let $\Pi_{p,h}^{per}: H^1_{per}(-2,2) \to \Sper_{p,h} $ be the $H^1_\emptyset$-orthogonal projection, where $\Sper_{p,h}$ are the periodic B-splines over $(-2,2)$ as defined in~\eqref{eq:sper}.
    This projector is by definition orthogonal with respect to the scalar product $(u,v)_{H^1_\emptyset(-2,2)} := (u', v')_{L^2(-2,2)} + (u,1)_{L^2(-2,2)}(v,1)_{L^2(-2,2)}$.
    \cite[Theorem~4.1]{Espen2019} states
    \[
        \| w - \Pi_{p,h}^{per} w \|_{L^2(-2,2)} \le \frac{h}{\pi} |w|_{H^1(-2,2)}
        \quad\mbox{for}\quad w\in H^1_{per}(-2,2).
    \]
    Let $v\in H^1_\Bdy(0,1)$ be arbitrary but fixed and let $w_h:=\Pi_{p,h}^{per} E_\Bdy v$.
    Analogous to the proof of~\cite[Theorem~1.1]{Takacs2016}, the construction guarantees
    \begin{equation}\label{eq:symmetry}
        (-1)^{\Bdy_0}(-1)^{\Bdy_1} w_h(x-2)
        = (-1)^{\Bdy_0} w_h(-x) 
        = w_h(x)
        = (-1)^{\Bdy_1} w_h(2-x) \quad\mbox{for all}\quad x\in [0,1].
    \end{equation}
    Using this observation, we immediately obtain $v_h := w_h|_{[0,1]}\in S_{p,h,\Bdy}^{reg}$ and $v_h = \Pi_{p,h,\Bdy} v$. From~\eqref{eq:wdef} and~\eqref{eq:symmetry}, we also obtain $\| w - w_h \|_{L^2(-2,2)} = 2 \| v - v_h \|_{L^2(0,1)} $. Since also $|w|_{H^1(-2,2)}=2|v|_{H^1(0,1)}$, we obtain the desired result~\eqref{eq:lem:approxerror2}.
\end{proof}

Now, we can give a proof of Theorem~\ref{thm:stable_splitting}.

\begin{proof}[Proof of Theorem~\ref{thm:stable_splitting}]
Let $Q_{p,h,\Bdy}: S_{p,h,\Bdy} \to S_{p,h,\Bdy}^{reg}$ be the $L^2$-orthogonal projection, so for any given $v_h\in S_{p,h,\Bdy}$, we have $v_h^{reg} = Q_{p,h,\Bdy} v_h$ and $v_h^{out} = (I-Q_{p,h,\Bdy}) v_h$.
Analogously to the \cite[proof of Theorem~6.1]{Takacs2016}, we obtain the inverse estimate
\begin{equation}\label{eq:inverse}
    | v_h |_{H^1(0,1)}
        \le \frac{2\sqrt{3}}{h} \|v_h\|_{L^2(0,1)}  \quad \mbox{for all} \quad v_h \in S_{p,h,\Bdy}^{reg}.
\end{equation}
Using the triangle inequality, the stability of the projector, \eqref{eq:inverse}, the fact that the $L^2$-projector minimizes the error in the $L^2$-norm and Lemma~\ref{lem:approxerror}, the desired result is proven analogously to~\cite{Hofreither2017}:
\begin{align*}
    |v_h^{reg}|_{H^1(0,1)}
    & = |Q_{p,h,\Bdy} v_h|_{H^1(0,1)}
    \le |\Pi_{p,h,\Bdy} v_h|_{H^1(0,1)} + |\Pi_{p,h,\Bdy} v_h - Q_{p,h,\Bdy} v_h|_{H^1(0,1)} \\
    & \le |v_h|_{H^1(0,1)} + 2\sqrt{3} h^{-1} \|\Pi_{p,h,\Bdy} v_h - Q_{p,h,\Bdy} v_h\|_{L^2(0,1)}
    \\
    & \le |v_h|_{H^1(0,1)} + 2\sqrt{3} h^{-1} \|v_h - \Pi_{p,h,\Bdy} v_h\|_{L^2(0,1)} + 2\sqrt{3} h^{-1} \|v_h - Q_{p,h,\Bdy} v_h\|_{L^2(0,1)}
    \\
    & \le |v_h|_{H^1(0,1)} + 4\sqrt{3} h^{-1} \|v_h - \Pi_{p,h,\Bdy} v_h\|_{L^2(0,1)} 
    \le \left(1 + \frac{4\sqrt{3}}{\pi}\right) |v_h|_{H^1(0,1)}. 
\end{align*}
All other statements follow immediately using the triangle and Young's inequalities.
\end{proof}

Finally, we show Lemma~\ref{lem:coefs}.

\begin{proof}[Proof of Lemma~\ref{lem:coefs}]
    Because of periodicity, we have 
    \begin{equation}
    \|F_j\|_{L^2(0,1)}^2 = \tfrac14 \|E_{\Bdy} F_j\|_{L^2(-2,2)}^2
    \quad\mbox{and}\quad
    |F_j|_{H^1(0,1)}^2 = \tfrac14 \|E_{\Bdy} F_j\|_{H^1(-2,2)}^2,
    \end{equation}
    where $E_{\Bdy}$ is as in~\eqref{eq:wdef}. By construction, $E_{\Bdy} F_j \in S_{p,h}^{per}$ and $E_{\Bdy} F_j(x_i)=\sin(x_i\alpha_j)$. 
    The cardinal B-spline functions
    \[
        \widetilde B_i(x):=\sum_{z\in\mathbb Z} \widetilde{\mathcal B}_p\left(\frac{x}{h}-i+4nz\right),
        \qquad i=-2n+1,\dots,2n
    \]
    form a basis of $S_{p,h}^{per}$.
    Let $M_{per}, K_{per}\in \mathbb R^{4n\times4n}$ be the mass and stiffness matrices for the periodic cardinal splines, as given by
    \[
        (M_{per})_{ij} = \int_{-2}^2 \widetilde B_j(x) \widetilde B_i(x) \mathrm d x
        \quad\mbox{and}\quad
        (K_{per})_{ij} = \int_{-2}^2 \widetilde B_j'(x) \widetilde B_i'(x) \mathrm d x,
    \]
    $C_{per}\in \mathbb R^{4n\times4n}$ the corresponding collocation matrix, as given by $(C_{per})_{ij} = \widetilde B_i(x_j)$ with $x_j$ as in~\eqref{eq:nodes}, and $Q_{per}\in \mathbb R^{4n\times4n}$ the Discrete Sine Transformation, as given by $(Q_{per})_{ij}=\sin(x_i\alpha_j)$, in all cases for $i,j=1,\dots,4n$.
    Analogous to the derivation of~\eqref{eq:Dreg} and~\eqref{eq:Lambda reg}, we have
    $\|F_j\|_{L^2(0,1)}^2 = (Q_{per}^\top C_{per}^{-\top} M_{per} C_{per}^{-1}Q_{per})_{jj}$
    and $|F_j|_{H^1(0,1)}^2 
       = (Q_{per}^\top C_{per}^{-\top} K_{per} C_{per}^{-1}Q_{per})_{jj}$.
    Since $C_{per}$, $M_{per}$ and $K_{per}$ are a symmetric circulant matrices, they are
    diagonalized using the Discrete Sine Transform matrix $Q_{per}$, see~\cite[Section~4.8]{GolubVanLoan}. So, we have
    \begin{align}\label{eq:lem:coefs:proof1}
       \|F_j\|_{L^2(0,1)}^2 
       = \frac{(Q_{per}^\top M_{per} Q_{per})_{jj}}{(Q_{per}^{\top} C_{per}Q_{per})_{jj}^2},
       \qquad
       |F_j|_{H^1(0,1)}^2 
       = \frac{(Q_{per}^\top K_{per} Q_{per})_{jj}}{(Q_{per}^\top C_{per}Q_{per})_{jj}^2}.
    \end{align}
    In order to compute the entries of the diagonal matrix obtained by the Discrete Sine Transform, which also called the symbol, we use that these diagonal entries are the eigenvalues corresponding to the eigenvectors $\mathbf q^{(j)}:=(\sin (x_{-2n+1}\alpha_j),\dots,\sin(x_{2n}\alpha_j))^\top$.
    So, we obtain
    In order to compute the terms in the denominator, we use this observation and the summation theorem for the sine and the symmetry of $C_{per}$ to obtain
    \begin{equation}\label{eq:lem:coefs:proof2}
    \begin{aligned}
        (C_{per}\mathbf q^{(j)})_i
            & = \sum_{l=-2n+1}^{2n} \widetilde B_i(x_l) \sin (x_l\alpha_j) 
             = \sum_{j=-2n+1}^{2n} \widetilde B_0(x_{l-i}) \cos((l-i)h\alpha_j) \sin (x_i\alpha_j) \\
            & = \sum_{k=-\lfloor p/2 \rfloor}^{\lfloor p/2 \rfloor} \widetilde{\mathcal B}_P(k) 
            \cos(kh\alpha_j) (\mathbf q^{(j)})_i,
    \end{aligned}
    \end{equation}
    which shows $(Q_{per}^\top C_{per}Q_{per})_{jj} = \sum_{k=-\lfloor p/2\rfloor}^{\lfloor p/2\rfloor} \widetilde{\mathcal B}_{p}(k) \cos (kh\alpha_j)$.
    In \cite[Section~3.5]{Takacs2016}, the symbol of the mass matrix is derived. It can be expressed using the Eulerian numbers or, equivalently, using the cardinal B-splines of degree $2p+1$, cf. also~\cite{WANG2010486}. \cite{Takacs2016}[Eqs.~(3.12), (3.17)] state that the symbol of the stiffness matrix can be expressed by a term $\frac{2-2\cos(h\alpha_j)}{h^2}$ modeling the derivative and the symbol of the mass matrix for degree $p-1$. So, we obtain
    \begin{equation}\label{eq:lem:coefs:proof3}
    \begin{aligned}
        (Q_{per}^\top M_{per} Q_{per})_{jj}
        &= \sum_{k=-p}^p \widetilde{\mathcal B}_{2p+1}(k) \cos (kh\alpha_j),
        \\
        (Q_{per}^\top K_{per} Q_{per})_{jj}
        &= \frac{2-2\cos(h\alpha_j)}{h} \sum_{k=-(p-1)}^{p-1} \widetilde{\mathcal B}_{2p-1}(k) \cos (kh\alpha_j).
    \end{aligned}
    \end{equation}
    The desired representations of $\left( D_{reg}\right)_{jj}$ and $\left( \Lambda_{reg} \right)_{jj}$ are obtained by combining~\eqref{eq:Dreg2}, \eqref{eq:Lambda reg2}, \eqref{eq:lem:coefs:proof1}, \eqref{eq:lem:coefs:proof2}, and \eqref{eq:lem:coefs:proof3}. 
\end{proof}

\textbf{Acknowledgements.} The research of the second author was funded by the Austrian Science Fund (FWF): 10.55776/P33956. The first and the third author are members of the Gruppo Nazionale Calcolo Scientifico-Istituto Nazionale di Alta Matematica (GNCS-INDAM). Moreover, the third author acknowledges support from the MUR through the PRIN 2022 PNRR project P2022NC97R (NOTES).

\bibliographystyle{siamplain}

\end{document}